\documentclass[11pt,reqno]{amsart}
\usepackage[T1]{fontenc}
\usepackage{lmodern}
\usepackage[margin=1in]{geometry}
\usepackage{amsmath,amssymb,mathtools}
\usepackage{mathrsfs}
\usepackage{enumitem}
\usepackage{needspace}
\usepackage{microtype}
\usepackage{xcolor}
\usepackage[unicode,colorlinks=true,linkcolor=blue!45!black,citecolor=blue!45!black,urlcolor=blue!45!black]{hyperref}
\hypersetup{pdftitle={Coanalytic Closures and Bases of Prescribed Hausdorff Dimension},pdfauthor={Yi Wang}}
\newtheorem{maintheorem}{Theorem}

\newtheorem{lemma}{Lemma}[section]
\newtheorem{theorem}[lemma]{Theorem}
\newtheorem{proposition}[lemma]{Proposition}
\newtheorem{corollary}[lemma]{Corollary}

\newtheorem*{claim*}{Claim}
\theoremstyle{definition}
\newtheorem{definition}[lemma]{Definition}
\newtheorem{convention}[lemma]{Convention}
\newtheorem{remark}[lemma]{Remark}
\newtheorem{question}[lemma]{Question}
\newcommand{\R}{\mathbb R}
\newcommand{\Q}{\mathbb Q}

\newcommand{\Z}{\mathbb Z}
\newcommand{\C}{\mathbb C}

\newcommand{\Cut}{\operatorname{Cut}}
\newcommand{\Dec}{\operatorname{Dec}}
\newcommand{\cl}{\operatorname{cl}}
\newcommand{\dcl}{\operatorname{dcl}}

\newcommand{\dimH}{\dim_{\mathrm H}}
\newcommand{\diam}{\operatorname{diam}}
\newcommand{\ran}{\operatorname{ran}}

\newcommand{\leqh}{\leq_{\mathrm h}}
\newcommand{\leqT}{\leq_{\mathrm T}}
\newcommand{\equivh}{\equiv_{\mathrm h}}

\newcommand{\Haus}{\mathcal H}

\DeclareMathOperator{\spanop}{span}
\DeclareMathOperator{\alg}{alg}

\DeclareMathOperator{\lab}{lab}
\DeclareMathOperator{\bat}{bat}
\DeclareMathOperator{\cell}{cell}
\numberwithin{equation}{section}
\setlist[enumerate]{leftmargin=*,itemsep=3pt}
\setlist[itemize]{leftmargin=*,itemsep=3pt}
\allowdisplaybreaks[1]
\title[Coanalytic closures and bases]{Coanalytic Closures and Bases\\of Prescribed Hausdorff Dimension}
\author{Yi Wang}
\address{Qiuzhen College, Tsinghua University, Beijing, China}
\date{}

\begin{document}
\begin{abstract}
Assume \(V=L\). We prove two uniform theorems producing coanalytic sets
of prescribed Hausdorff dimension in Euclidean space. The first concerns
countable families of partial operations of finite arity on \(\R^n\)
whose evaluation is hyperarithmetic and which are tame in each vector
variable separately. For such families a single lightface \(\Pi^1_1\) relation
yields strictly increasing, left-continuous chains of proper sets closed
under the operations, one of each Hausdorff dimension in \([0,n]\),
uniformly in the family and in \(n\); no exchange property of the
closure is needed. As an application, there are coanalytic proper
subfields of \(\R\) of every Hausdorff dimension in \([0,1]\). The second theorem constructs
coanalytic bases of every possible Hausdorff dimension for effectively
coded pregeometries. As an application, there are coanalytic Hamel bases of \(\R^n\)and transcendence bases of \(\C\) of every Hausdorff dimension in \([0,n]\) and \([0,2]\).
\end{abstract}
\maketitle
\tableofcontents

\section{Introduction}

Let \(n\geq1\), and let \(\mathcal F\) be a countable collection of
partial operations \((\R^n)^m\rightharpoonup\R^n\), each of some finite
arity \(m\). A set \(A\subseteq\R^n\) is closed under \(\mathcal F\) if
it contains \(f(\bar a)\) whenever \(f\in\mathcal F\), \(\bar a\) is a
tuple from \(A\), and \(f(\bar a)\) is defined. We ask which Hausdorff
dimensions a proper set of this kind can have, and how simple its
definition can be. Many familiar structures arise in this way. The field
operations give subfields of \(\R\) or of \(\C\), and adding functions
that select roots of polynomials gives real closed or algebraically
closed subfields. More generally, the definable functions of an
o-minimal expansion of the real field are closed under composition, and
the sets closed under them are the definably closed sets.

Our first main theorem answers this question under \(V=L\) for a broad
class of effectively presented families, in which the operations and
the ambient dimension \(n\) are both allowed to vary. Its hypotheses are
of two kinds, and we keep them apart. The effective hypothesis is that
the evaluation relation is hyperarithmetic; the geometric one asks that
every operation, viewed as a function of a single vector argument with
the remaining arguments fixed, be undefined, constant, or a local
bi-Lipschitz homeomorphism on an open neighborhood of a subcube of
proportional size in every cube. Under these hypotheses we obtain, for every
\(\alpha\in[0,n]\), a proper set of dimension \(\alpha\) closed under
the operations, and a single coanalytic relation defines these sets
jointly in the family, the dimension and the point. The construction
controls the whole generated closure, not merely a set of generators,
and it needs no exchange property of the closure operator.
Our second main theorem treats a related problem for bases: for suitable
effectively coded pregeometries, it constructs coanalytic bases of every
dimension allowed by the ambient Euclidean space.

The field case has a classical motivation. Erd\H{o}s and Volkmann
\cite{EV66} constructed Borel additive subgroups of \(\R\) of every
Hausdorff dimension in \([0,1]\), and Erd\H{o}s asked whether the same
dimensions could be realized by subrings or subfields
\cite[p.~119]{Erd79}; see also Problem~1154 in \cite{Blo}.
Multiplication creates a substantial obstruction at low descriptive
complexity. Edgar and Miller
\cite[Theorem~1 and the remarks on p.~1124]{EM03} proved that an analytic
subring \(S\subseteq\R\) satisfies
\[
                 \dimH S=0\quad\hbox{or}\quad S=\R,
\]
and Bourgain \cite{Bou03} independently excluded Borel subrings of
intermediate Hausdorff dimension. Under the continuum hypothesis,
Mauldin \cite{Mau16} constructed subfields of every dimension in
\([0,1]\); his final remark observes that under \(V=L\) the fields
can be chosen \(\Sigma^1_2\).

Applied to the field operations, Theorem~\ref{thm:A} gives, under
\(V=L\), a strictly increasing chain of proper subfields of \(\R\), one
of each Hausdorff dimension in \([0,1]\), all defined by a single
lightface \(\Pi^1_1\) relation (Corollary~\ref{cor:subfields}). This
improves the \(\Sigma^1_2\) bound in Mauldin's remark, and by the
theorem of Edgar and Miller coanalytic is the optimal complexity.
Richer families of operations give stronger closure properties: the
subfields can be taken to be real closed, and even to be elementary
substructures of the real exponential field
(Corollary~\ref{cor:exponential}).

For bases, the dimension question is already nontrivial without a
definability requirement. Lutz, Qi and Yu \cite{LQY24} proved in ZFC
that a Hamel basis of \(\R\) over \(\Q\) can have any Hausdorff
dimension in \([0,1]\).
Theorem~\ref{thm:B} provides coanalytic bases of every dimension under
\(V=L\), uniformly over a general family of pregeometries and the
prescribed dimension.

No Hamel basis of \(\R^n\), \(n\ge1\), over a countable subfield of
\(\R\), and no transcendence basis of \(\C\) over a countable subfield
of \(\C\), is analytic \cite[Theorems~3.2 and~5.2]{WZ12}.
Thus the coanalytic complexity is optimal for these real-field and
basis applications, apart from the countable zero-dimensional field sections.
We make no such claim for the general function families of
Theorem~\ref{thm:A}.

\begin{definition}[Coded function families]\label{def:function-family}
A coded function family consists of a lightface \(\Delta^1_1\) set
\(I\subseteq(\omega\setminus\{0\})\times2^\omega\) and, for every
\(\iota=(n,\gamma)\in I\), a countable family
\(\mathcal F_\iota=\{f_{\iota,e}:e\in\omega\}\) of partial maps
\((\R^n)^m\rightharpoonup\R^n\), \(m\geq0\), with the following
properties.
\begin{enumerate}[label=\textup{(F\arabic*)}]
\item \(\mathcal F_\iota\) contains the projections
      \((x_1,\ldots,x_m)\mapsto x_i\) and, as nullary maps, the constants
      \(q\in\Q^n\). It is closed under composition, under permutation
      and identification of variables, and under the addition of dummy
      variables. (This is a requirement on the family as a set of maps;
      indices of compositions need not be computable.)
\item The arity of \(f_{\iota,e}\) is a computable function of \(e\),
      and the evaluation relation
      \[
       \operatorname{Eval}(n,\gamma,e,\bar a,y)
       \quad\Longleftrightarrow\quad f_{\iota,e}(\bar a)=y
      \]
      is lightface \(\Delta^1_1\).
\item For every map in \(\mathcal F_\iota\) and every choice of all but
      one of its vector arguments, the resulting section
      \(h(x)=f_{\iota,e}(\bar a,x)\), \(x\in\R^n\), admits
      \(\varepsilon_h>0\) such that every compact axis-parallel cube
      \(Q\subseteq\R^n\) of positive side length contains such a cube
      \(Q'\) satisfying \(\ell(Q')\geq\varepsilon_h\ell(Q)\) and
      admitting an open neighborhood \(V\) with one of the following
      properties:
      \begin{enumerate}[label=\textup{(\alph*)}]
      \item \(h\) is undefined throughout \(V\);
      \item \(h\) is defined on all of \(V\) and \(h|_V\) is constant;
      \item \(h\) is defined on all of \(V\) and \(h|_V\) is locally
            a bi-Lipschitz homeomorphism onto an open set.
      \end{enumerate}
      Here \(\ell(Q)\) denotes the side length of \(Q\).
\end{enumerate}
Neither the constants \(\varepsilon_h\) nor the neighborhoods are
required to be effective. Define
\[
 \cl_\iota(A)=
 \{f(\bar a):f\in\mathcal F_\iota,\ \bar a\in A^{<\omega},\
                       f(\bar a)\text{ is defined}\}.
\]
The family may be given relative to a fixed real parameter \(\delta\),
in which case the effective requirements are relativized to \(\delta\).
\end{definition}

\begin{maintheorem}[Uniform coanalytic closures]\label{thm:A}
Assume \(V=L\), and let \((\mathcal F_\iota)_{\iota\in I}\) be a
coded function family. There is a lightface \(\Pi^1_1\) set
\[
 G\subseteq
 \{(n,\gamma,\alpha,y):(n,\gamma)\in I,\ 0\leq\alpha\leq n,\
                                             y\in\R^n\}
\]
such that for every \(\iota=(n,\gamma)\in I\) the sections
\[
 A_\iota(\alpha)=\{y\in\R^n:(n,\gamma,\alpha,y)\in G\},
 \qquad 0\leq\alpha\leq n,
\]
have the following properties.
\begin{enumerate}[label=\textup{(\arabic*)}]
\item Each \(A_\iota(\alpha)\) is a proper subset of \(\R^n\) closed
      under \(\mathcal F_\iota\), and \(A_\iota(0)=\cl_\iota(\varnothing)\).
\item \(\dimH A_\iota(\alpha)=\alpha\).
\item \(A_\iota(\alpha)\subsetneq A_\iota(\beta)\) whenever
      \(\alpha<\beta\), and
      \[
       A_\iota(\alpha)=
       \bigcup_{\substack{r\in\Q\\0<r<\alpha}}A_\iota(r)
       \qquad(\alpha>0).
      \]
\end{enumerate}
Since \(G\) is lightface, each \(A_\iota(\alpha)\) is
\(\Pi^1_1(\gamma,\alpha)\). If the family is given relative to a real
parameter \(\delta\), the same holds with \(G\in\Pi^1_1(\delta)\).
\end{maintheorem}

The set \(A_\iota(0)=\cl_\iota(\varnothing)\) is countable and contains
\(\Q^n\), and by (2) it is the only countable member of the chain. If
every \(A_\iota(\alpha)\) is to contain a prescribed countable set of
parameters, it suffices to add these parameters to the family as
nullary maps.

Condition (F3) is the only geometric hypothesis. In dimension one it
holds for a partial map \(h\) whenever \(\R\) can be partitioned into
finitely many points and open intervals on each of which \(h\) is
undefined, constant, or \(C^1\) with nowhere vanishing derivative; this
is the case for rational functions and, more generally, for the
definable functions of an o-minimal expansion of the real field, see
Lemma~\ref{lem:one-variable} and the proof of
Corollary~\ref{cor:exponential}. In any dimension a
\(C^1\) map with nonsingular derivative is locally a bi-Lipschitz
homeomorphism onto an open set, by the inverse function theorem, but differentiability
is not part of the hypotheses. Note also that \textup{(a)--(c)} concern the
whole vector variable \(x\in\R^n\): a section of intermediate rank, such as
\((x_1,x_2)\mapsto(x_1,0)\) on \(\R^2\), is neither locally constant nor
locally invertible, and (F3) fails for it.

Take \(n=1\), and let \(\mathcal F\) consist of the partial maps
\(\R^m\rightharpoonup\R\), \(m\ge0\), given by fractions \(P/Q\), where
\(P,Q\in\Q[x_1,\ldots,x_m]\), each defined exactly where its displayed
denominator \(Q\) is nonzero. The case \(Q\equiv0\) gives the
nowhere-defined map. A subset of \(\R\) is closed under \(\mathcal F\)
exactly when it is a subfield,
and \(\cl(\varnothing)=\Q\). The hypotheses of Theorem~\ref{thm:A} are
checked at the end of Section~\ref{sec:general-closures}; for (F3), a unary
section is either nowhere defined or a rational function of one variable
on the complement of finitely many points.

\begin{corollary}[Subfields of every dimension]\label{cor:subfields}
Assume \(V=L\). There is a lightface \(\Pi^1_1\) set
\(G_{\mathrm{fld}}\subseteq[0,1]\times\R\) such that the sections
\[
 F(\alpha)=\{y\in\R:(\alpha,y)\in G_{\mathrm{fld}}\},
 \qquad 0\leq\alpha\leq1,
\]
are proper subfields of \(\R\) with \(\dimH F(\alpha)=\alpha\).
Moreover, \(F(0)=\Q\), and \(F(\alpha)\subsetneq F(\beta)\) for
\(\alpha<\beta\), with \(F(\alpha)=\bigcup_{r\in\Q,\,0<r<\alpha}F(r)\)
for \(\alpha>0\).
\end{corollary}

\Needspace{14\baselineskip}
\begin{corollary}[Elementary substructures of the real exponential field]
\label{cor:exponential}
Assume \(V=L\), and let
\(\R_{\exp}=(\R;<,+,-,\cdot,0,1,\exp)\).
There is a lightface \(\Pi^1_1\) set
\(G_{\exp}\subseteq[0,1]\times\R\) such that the sections
\[
 E(\alpha)=\{y\in\R:(\alpha,y)\in G_{\exp}\},
 \qquad 0\leq\alpha\leq1,
\]
are proper elementary substructures of \(\R_{\exp}\), with
\(\dimH E(\alpha)=\alpha\) and
\(E(0)=\dcl_{\R_{\exp}}(\varnothing)\).
Moreover, \(E(\alpha)\subsetneq E(\beta)\) whenever
\(\alpha<\beta\), and
\[
 E(\alpha)=
 \bigcup_{\substack{r\in\Q\\0<r<\alpha}}E(r)
 \qquad(\alpha>0).
\]
In particular, each \(E(\alpha)\) is a real closed subfield of \(\R\).
\end{corollary}

The second theorem addresses bases. A pregeometry is a closure operator
with finite character and exchange; a basis is an independent set whose
closure is the whole space. Here the hypotheses concern countable
closures, addition, and effective membership in a finite closure.

\begin{definition}[Coded families of pregeometries]\label{def:pregeometry-family}
A coded family of pregeometries consists of a lightface \(\Delta^1_1\) set
\(I\subseteq(\omega\setminus\{0\})\times2^\omega\) and, for each
\(\iota=(k,\gamma)\in I\), a pregeometry \(\cl_\iota\) on \(\R^k\), subject to:
\begin{enumerate}[label=\textup{(C\arabic*)}]
\item \(\cl_\iota(A)\) is countable whenever \(A\subseteq\R^k\) is countable;
\item \(a+b\in\cl_\iota(\{a,b\})\) for all \(a,b\in\R^k\);
\item the relation
      \[
       (k,\gamma)\in I\quad\hbox{and}\quad
       y\in\cl_{(k,\gamma)}(\ran\bar x)
      \]
      is lightface \(\Delta^1_1\), uniformly in \(k,\gamma\), the finite tuple
      \(\bar x\in(\R^k)^{<\omega}\), and \(y\in\R^k\).
\end{enumerate}
\end{definition}

\begin{maintheorem}[Bases of pregeometries]\label{thm:B}
Assume \(V=L\), and let \((\cl_\iota)_{\iota\in I}\) be a coded family of pregeometries.
There is a lightface \(\Pi^1_1\) set
\[
 \mathcal G\subseteq
 \{(k,\gamma,s,x):(k,\gamma)\in I,\ 0\le s\le k,\ x\in\R^k\}
\]
such that every section
\[
 \mathcal G_{k,\gamma,s}=
 \{x:(k,\gamma,s,x)\in\mathcal G\}
\]
is a basis of \(\cl_{(k,\gamma)}\) and has Hausdorff dimension \(s\).
Each section is \(\Pi^1_1(\gamma,s)\). If the family is given relative
to a real parameter \(\delta\), the same holds with
\(\mathcal G\in\Pi^1_1(\delta)\).
\end{maintheorem}

\begin{corollary}[Hamel and transcendence bases]\label{cor:bases}
Assume \(V=L\). There are lightface \(\Pi^1_1\) relations
\(H(n,s,x)\) and \(T(s,z)\) such that:
\begin{enumerate}[label=\textup{(\arabic*)}]
\item For every \(n\ge1\) and \(s\in[0,n]\), the set
      \[
       H_{n,s}=\{x\in\R^n:H(n,s,x)\}
      \]
      is a Hamel basis of \(\R^n\) over \(\Q\) with
      \(\dimH H_{n,s}=s\).
\item For every \(s\in[0,2]\), the set
      \[
       T_s=\{z\in\C:T(s,z)\}
      \]
      is a transcendence basis of \(\C\) over \(\Q\) with
      \(\dimH T_s=s\).
\end{enumerate}
Each section \(H_{n,s}\) and \(T_s\) is \(\Pi^1_1(s)\).
\end{corollary}

The definability method originates in work of Erd\H{o}s, Kunen and
Mauldin \cite{EKM81} and Miller \cite{Mil89}. We use Vidny\'anszky's
selection theorem \cite{Vid14}: under \(V=L\), a coanalytic recursion
whose admissible choices are cofinal in the hyperarithmetic degrees has
a coanalytic set of construction records. This selection step is the
only use of \(V=L\).

A closely related method appears in Richter's construction of coanalytic
counterexamples to Marstrand's projection theorem \cite{Ric25}, where
sparse digit coding controls countably many scalar images within a
coanalytic recursion. Our Cantor constructions use a related strategy.
For Theorem~\ref{thm:A}, we additionally arrange recovery of
finite-support construction records from generated values outside the
closure of the empty set. This recovery permits a coanalytic definition
of the whole generated closure.

In the proof of Theorem~\ref{thm:A} the sets \(A_\iota(\alpha)\) are
generated by points chosen one at a time, and each generator \(z\)
carries a rational tag \(r\), the dimension it is allowed to
contribute. The new generator is taken from a Cantor set built so as to
control, simultaneously, the countably many unary sections
\(h(x)=f(\bar a,x)\) with \(\bar a\) from the earlier generators. By
(F3) we may shrink the cubes of the construction until each section is
undefined, constant, or locally invertible near every point of the
Cantor set. A locally constant section contributes nothing new: its
value is obtained by substituting a rational vector, so it lies in the
closure of \(\bar a\) already. On a locally invertible branch, the
\(\Delta^1_1\) graph of the section lets us recover \(z\)
hyperarithmetically from \(h(z)\) and \(\bar a\); the geometric
witnesses never have to be computed.

The Cantor set is chosen so that every value which is new over some
finite tuple of earlier generators lies in an approximation set
\(E_n(r)\) of dimension at most \(r\), and encodes the history of the
construction. We impose this over every finite tuple, and not merely
over the whole earlier closure, because the definability argument
proceeds through minimal finite supports: if \(y\) is generated by a
minimal finite set of generators, then the last of them, together with
the history before it, is hyperarithmetic in \(y\). A
hyperarithmetically bounded existential quantifier then gives a
\(\Pi^1_1\) definition of the generated closure. Only finite character
and the recovery of local inverses are used here, which is why no
exchange property is needed.

The lower bound on the dimension comes from dimension demands: for
every rational \(r\) and every coded set of dimension less than \(r\),
some generator of tag \(r\) lies outside that set. A separate
reservation demand chooses, for each \(\iota\), a point that is kept
out of the closure at all later stages; this makes the sets proper
even when \(\alpha=n\).
Finally, \(A_\iota(\alpha)\) is the closure of the generators with tag
below \(\alpha\), which gives monotonicity, left continuity and
\(A_\iota(0)=\cl_\iota(\varnothing)\) at once.

For Theorem~\ref{thm:B}, a spanning demand for \(y\in\R^k\) is met,
when necessary, by adjoining a pair \(\{x,y-x\}\) from a fixed
zero-dimensional approximation set. The addition hypothesis ensures
that the pair spans \(y\), and exchange ensures independence.
Both points compute the spanning target and are hyperarithmetically
equivalent, so either point recovers the finite batch. Together with
the label, it recovers the corresponding construction record.
Further demands avoid all coded sets of smaller dimension. Thus the basis construction
has its own algebraic hypotheses and does not require the geometric
condition \textup{(F3)} used for closures.

Section~\ref{sec:preliminaries} fixes the effective and geometric
conventions. Sections~\ref{sec:cantor-sets} and~\ref{sec:null-sets} supply
the coded Cantor sets and null sets used in both constructions.
Section~\ref{sec:general-closures} proves Theorem~\ref{thm:A} and
Corollary~\ref{cor:subfields}, and gives further field applications.
Section~\ref{sec:bases} proves Theorem~\ref{thm:B} and verifies its Hamel- and
transcendence-basis applications. The final section discusses descriptive
optimality, contrasting consistency results, and further questions.

\section{Preliminaries}\label{sec:preliminaries}

Elements of \(\omega^\omega\) and of \(2^\omega\subseteq\omega^\omega\) are called
reals. Fix a computable pairing function \(\langle\cdot,\cdot\rangle:\omega^2\to\omega\).
For \(a\in\omega^\omega\), its \(n\)-th column is
\(a^{[n]}(i)=a(\langle n,i\rangle)\).
Finite and countable joins are denoted by \(\oplus\); in a finite join all
unused columns are zero. Write \(a^+(i)=a(i+1)\), and let \(e^\frown u\) have
first entry \(e\) and tail \(u\). The constant sequence with value \(k\) is
denoted by \(\widehat k\).

Fix a computable enumeration \((\rho_i)_{i\in\omega}\) of \(\Q\), without
repetitions. For \(x\in\R\), define
\[
 \Cut(x)(i)=1\quad\Longleftrightarrow\quad \rho_i<x.
\]
For \(x=(x_1,\ldots,x_k)\in\R^k\), put
\(\Cut(x)=\Cut(x_1)\oplus\cdots\oplus\Cut(x_k)\).
Identify \(\C\) with \(\R^2\).
Scalar cuts and the codes of one-coordinate vectors are computably
interconvertible; in constructions uniform in \(k\), we use the vector code
also when \(k=1\).
Let \(\Dec(v)\) be the real whose scalar cut is \(v\), if one exists, and let
\(\Dec(v)=0\) otherwise.

\begin{convention}\label{conv:codes}
In tuples involving Euclidean points, replace each such point by its cut.
Lightface pointclasses always refer to the resulting real codes.
Thus \(\Pi^1_1(\gamma,s)\) means \(\Pi^1_1(\gamma\oplus\Cut(s))\).
When \(x\in\R^k\) occurs in a reducibility statement, it stands for \(\Cut(x)\).
This oracle contains all coordinate cuts; an individual coordinate is
written explicitly as \(\Cut(x_i)\).
For a label \(\iota=(n,\gamma)\), use the real code
\(\widehat n\oplus\gamma\).
Finite tuples are coded by finite joins together with their lengths.
\end{convention}

\begin{lemma}[Arithmetic operations on cuts]\label{lem:cuts}
\begin{enumerate}[label=\textup{(\alph*)}]
\item The set \(\{(k,\Cut(x)):k\ge1,\ x\in\R^k\}\) is arithmetic.
\item For \(x\in\R\), a cut \(\Cut(x)\), together with \(n\in\omega\), uniformly computes
      rationals \(a<x\le b\) with \(b-a\le2^{-n}\).
\item For a rational polynomial \(P\) and a real tuple \(\bar x\), the relation
      \(P(\bar x)>0\) is \(\Sigma^0_1\), and \(P(\bar x)=0\) is \(\Pi^0_1\), in
      \(\Cut(\bar x)\), uniformly in \(P\). Consequently
      \(\Cut(P(\bar x))\) is uniformly computable in \(\Cut(\bar x)'\).
\item If \(\cell_N(x)\) is the unique integer \(v\) satisfying
      \[
                  v2^{-N}<x\le(v+1)2^{-N},
      \]
      then \(\cell_N(x)\) is uniformly computable from \(N\) and \(\Cut(x)\).
\end{enumerate}
\end{lemma}

\begin{proof}
Scalar cuts are exactly the codes for nonempty proper downward-closed
subsets of \(\Q\) with no largest element. These conditions are arithmetic.
For vector cuts apply them to the first \(k\) columns and require the
remaining columns to vanish.

A cut decides each comparison \(q<x\). Find a rational bracket for \(x\)
and bisect it to prove (b). Interval arithmetic then gives rational intervals
containing \(P(\bar x)\), with lengths tending effectively to zero.
Positivity is witnessed when one of these intervals has positive left
endpoint; equality to zero holds exactly when all contain zero.
Apply the positivity test to \(P-\rho_i\) to obtain (c).
Finally, first find an integer bound on \(|x|\), then perform a finite search
for the largest integer \(v\) satisfying \(v2^{-N}<x\), proving (d).
\end{proof}

A scalar cut computes a Cauchy name by Lemma~\ref{lem:cuts}(b), and a
vector cut does so coordinatewise, uniformly in the ambient dimension.
Conversely, strict comparison with a rational is semidecidable from a
Cauchy name, and one jump computes the whole cut. Thus these representations give
the usual effective projective pointclasses on Euclidean spaces, in particular
\(\Delta^1_1\) and \(\Pi^1_1\); see \cite{Mos09}.
For an individual computable real the cut is computable: use direct
comparison if it is rational, and successive approximation if it is
irrational.

Write \(a\leqh b\) if \(a\) is hyperarithmetic in \(b\), equivalently if
\(a\in\Delta^1_1(b)\), and write \(a\equivh b\) for mutual reducibility.
Finite joins and finite Turing jumps of reals hyperarithmetic in \(b\)
are again hyperarithmetic in \(b\).
We use the following standard effective facts \cite{Sac90,Mos09,Mil95,Vid14}.

\begin{lemma}\label{fact:effective}
\begin{enumerate}[label=\textup{(\alph*)}]
\item The relation \(a\leqh b\) is lightface \(\Pi^1_1\).
\item If \(P(v,b)\) is \(\Pi^1_1(t)\), then
      \[
                   (\exists b\leqh v)\,P(v,b)
      \]
      is \(\Pi^1_1(t)\) as a condition on \(v\).
\item If a \(\Sigma^1_1(d)\) set of reals is countable, every member is
      hyperarithmetic in \(d\). In particular, the unique member of a
      \(\Delta^1_1(d)\) singleton is hyperarithmetic in \(d\).
\item The classes \(\Pi^1_1\) and \(\Delta^1_1\) are closed under number
      quantifiers and finite conjunctions and disjunctions. The class
      \(\Delta^1_1\) is also closed under complements. These statements
      relativize uniformly.
\end{enumerate}
\end{lemma}

Part (b) is the Spector--Gandy theorem.
Part (c) is the effective perfect set theorem; a convenient statement is
\cite[Theorem~10.21]{Mar13}.
A useful consequence is that a partial function with
\(\Delta^1_1(d)\) graph has values hyperarithmetic in \(d\) and its input,
wherever it is defined.
This applies equally to partial maps between Euclidean spaces, with the
input and output represented by their whole-vector cuts.

A \emph{history} is a finite or countably infinite sequence of reals.
Fix the usual recursive coding of the space
\((\omega^\omega)^{\le\omega}\) by reals, so that its length and its entries
are uniformly available from its code.
For a history \(a\), let \(\ran a=\{a(n):n<|a|\}\).

\begin{definition}\label{def:compatible}
A set \(Y\subseteq\omega^\omega\) is compatible with
\[
 \mathscr F\subseteq(\omega^\omega)^{\le\omega}
                        \times\omega^\omega\times\omega^\omega
\]
if there are enumerations
\(\omega^\omega=\{p_\xi:\xi<\omega_1\}\) and
\(Y=\{b_\xi:\xi<\omega_1\}\), and histories \(a_\xi\) with
\(\ran a_\xi=\{b_\eta:\eta<\xi\}\), such that
\((a_\xi,p_\xi,b_\xi)\in\mathscr F\) for every \(\xi<\omega_1\).
A set \(X\) is cofinal in the hyperdegrees if, for every \(u\), some
\(x\in X\) satisfies \(u\leqh x\).
\end{definition}

\begin{theorem}[Vidny\'anszky's selection theorem]\label{fact:selection}
Assume \(V=L\).
Let \(\mathscr F\subseteq(\omega^\omega)^{\le\omega}
\times\omega^\omega\times\omega^\omega\) be \(\Pi^1_1(t)\).
If every section
\(\mathscr F_{a,p}=\{b:(a,p,b)\in\mathscr F\}\)
is cofinal in the hyperdegrees, there is a \(\Pi^1_1(t)\) set compatible
with \(\mathscr F\).
\end{theorem}

This is \cite[Theorem~3.4]{Vid14}.
The selection relations used below are lightface unless a fixed real
parameter is explicitly introduced.

Use the Euclidean norm \(|\cdot|\) and the maximum norm
\(\|\cdot\|_\infty\) on \(\R^k\).
All cubes are axis-parallel with positive side length.
For \(A\subseteq\R^k\), \(t>0\), and \(\delta>0\), put
\[
 \Haus_\delta^t(A)=
 \inf\left\{\sum_{i\in\omega}(\diam V_i)^t:
 A\subseteq\bigcup_iV_i,\ \diam V_i\le\delta\right\},
 \qquad
 \Haus^t(A)=\lim_{\delta\downarrow0}\Haus_\delta^t(A).
\]
The Hausdorff dimension is
\(\dimH A=\inf\{t>0:\Haus^t(A)=0\}\), with \(\dimH\varnothing=0\).
We recall the elementary facts needed in the constructions; see \cite{Fal14}.

\begin{lemma}\label{lem:hausdorff}
Let \(A,A_0,A_1,\ldots\subseteq\R^k\).
\begin{enumerate}[label=\textup{(\alph*)}]
\item If \(t>\dimH A\), then \(\Haus^t(A)=0\); if
      \(0<t<\dimH A\), then \(\Haus^t(A)>0\).
\item For \(t>0\), \(\Haus^t\) is countably subadditive and vanishes on
      countable sets. In particular,
      \[
           \dimH\bigcup_n A_n=\sup_n\dimH A_n.
      \]
\item Suppose \(K\subseteq\R^k\) is compact and \(\mu\) is an atomless Borel
      probability measure supported on \(K\).
      If \(\mu(Q)\le C\lambda^t\) for every closed cube \(Q\) of side
      \(\lambda>0\), then \(\Haus^t(K)\ge C^{-1}\).
\end{enumerate}
\end{lemma}

\begin{proof}
The monotonicity of Hausdorff measures in the exponent gives (a), and
countable subadditivity gives (b).
For (c), any set \(V\) of positive diameter \(\lambda\) lies in a closed
cube of side \(\lambda\), so
\(\mu^*(V)\le C(\diam V)^t\).
Sets of diameter zero have zero mass.
For a countable cover of \(K\), summing these inequalities gives
\(1\le C\sum_i(\diam V_i)^t\), and taking infima proves the assertion.
\end{proof}

A pregeometry on a set \(X\) is an extensive, monotone, idempotent map
\(\cl:\mathcal P(X)\to\mathcal P(X)\) with finite character and exchange:
\[
 a\in\cl(A\cup\{b\})\setminus\cl(A)
       \quad\Longrightarrow\quad b\in\cl(A\cup\{a\}).
\]
A set \(A\) is independent if \(a\notin\cl(A\setminus\{a\})\) for every
\(a\in A\), and is a basis if it is independent and \(\cl(A)=X\).

\begin{lemma}\label{lem:independence}
\begin{enumerate}[label=\textup{(\alph*)}]
\item A set is independent if and only if every finite subset is independent.
\item If \(A\) is independent and \(b\notin\cl(A)\), then \(A\cup\{b\}\) is independent.
\item The union of a chain of independent sets is independent.
\end{enumerate}
\end{lemma}

\begin{proof}
Finite character proves (a), and then (c).
For (b), if some \(a\in A\) belonged to
\(\cl((A\setminus\{a\})\cup\{b\})\), exchange and the independence of
\(A\) would imply \(b\in\cl(A)\), a contradiction.
\end{proof}

\section{Coded Cantor sets}
\label{sec:cantor-sets}

This section provides the geometric constructions used in both main
theorems.  The first construction controls the images of a Cantor set
under countably many partial maps of a Euclidean space.  The second controls
a point and its reflection in every Euclidean dimension.  In both
constructions, sparse binary digits of the images encode a prescribed
real without changing the desired Hausdorff dimension.

\begin{definition}\label{def:approximation}
For an integer $k\geq 1$ and a real $t>0$, let $E_k(t)$ be the set of
$x\in\R^k$ such that, for infinitely many integers $n\geq 1$, there is
$v\in\mathbb Z^k$ with
\[
   \|x-2^{-n}v\|_\infty\leq 2^{-\lceil kn/t\rceil}.
\]
For $s\in[0,k]$, put
\[
   U_k(s)=\bigcap\{E_k(t):t\in\Q,\ t>s\}.
\]
\end{definition}

\begin{lemma}\label{lem:approximation}
The following statements hold for every $k\geq 1$.
\begin{enumerate}[label=\textup{(\alph*)}]
\item If $t\geq k$, then $E_k(t)=\R^k$.  In particular,
      $U_k(k)=\R^k$.
\item If $0<t\leq t'$, then $E_k(t)\subseteq E_k(t')$.  If
      $0\leq s\leq s'\leq k$, then $U_k(s)\subseteq U_k(s')$.
\item The relations
\[
 \{(k,t,x):k\geq1,\ t\in\Q_{>0},\ x\in E_k(t)\}
 \quad\hbox{and}\quad
 \{(k,s,x):k\geq1,\ s\in[0,k],\ x\in U_k(s)\}
\]
      are arithmetic under the fixed coding by cuts.
\item If $0<t<k$ and $t'>t$, then
      $\mathcal H^{t'}(E_k(t))=0$.  Consequently
\[
       \dimH E_k(t)\leq t \quad(t>0),
       \qquad \dimH U_k(s)\leq s \quad(0\leq s\leq k).
\]
\end{enumerate}
\end{lemma}

\begin{proof}
If $t\geq k$, then $\lceil kn/t\rceil\leq n$.  Rounding each
coordinate of $2^n x$ to a nearest integer gives
$\|x-2^{-n}v\|_\infty\leq2^{-n-1}$, proving the first assertion.
The inclusions in the second assertion follow directly from the
monotonicity of the exponents and the definition of $U_k(s)$.

For rational $t>0$, the integer $\lceil kn/t\rceil$ is computable from
$k,n,t$.  Comparisons of a coordinate of $x$ with a rational number
are arithmetic in its cut.  The displayed approximation condition,
with its integer quantifiers over $n$ and $v$, is therefore arithmetic,
uniformly in $k,t,x$.  Also, $s<t$ is an arithmetic condition in
$\Cut(s)$ and the rational $t$.  Hence
\[
 x\in U_k(s)
 \quad\Longleftrightarrow\quad
 (\forall t\in\Q_{>0})\,[s<t\ \Longrightarrow\ x\in E_k(t)]
\]
is arithmetic, proving the third assertion.

To prove the last assertion, fix $M\in\mathbb N$, $0<t<k$, and
$t'>t$.  At approximation level $n$, the relevant
points in $[-M,M]^k$ are covered by cubes of side
$2\cdot2^{-\lceil kn/t\rceil}$ centered at $2^{-n}v$, with
$\|2^{-n}v\|_\infty\leq M+1$.  There are at most
\[
 (2^{n+1}(M+1)+1)^k\leq (2M+3)^k2^{kn}
\]
such centers, and each cube has Euclidean diameter at most
$2\sqrt{k}\,2^{-kn/t}$.  For every $n_0$, the cubes at levels
$n\geq n_0$ cover $E_k(t)\cap[-M,M]^k$, and their total
$t'$-cost is at most
\[
 (2M+3)^k(2\sqrt{k})^{t'}
 \sum_{n\geq n_0}2^{kn(1-t'/t)}.
\]
This tends to zero as $n_0\to\infty$, and so do the diameters of the
cubes.  Thus $\mathcal H^{t'}(E_k(t)\cap[-M,M]^k)=0$.
Taking the countable union over $M$ proves the assertion about
$E_k(t)$.  For $t\geq k$, its dimension bound is immediate.
If $s<k$, then $U_k(s)\subseteq E_k(t)$ for every rational
$t\in(s,k)$, whence $\dimH U_k(s)\leq s$.  The case $s=k$ follows
from the ambient dimension.
\end{proof}

Fix an integer $k\geq1$.  Lipschitz estimates use the maximum norm.
For a partial map $f:D_f\to\R^k$, we use alternatives
\textup{(a)--(c)} of \textup{(F3)}, with $f$ in place of $h$.
An alternative holds at $x$ if it holds on some open neighborhood
of $x$. At most one alternative can hold at a given point.

We say that $f$ satisfies \emph{quantitative trimming}, abbreviated
\textup{(QT)}, if there is $\varepsilon_f>0$ such that every
nondegenerate compact cube $P$ contains a compact cube $Q$ with
$\ell(Q)\geq\varepsilon_f\ell(P)$ that has an open neighborhood
on which one of \textup{(a)--(c)} holds for $f$. Thus condition (F3) of
Definition~\ref{def:function-family} says that every unary section of
a map in the family satisfies \textup{(QT)}. Note that $f$ need not be
injective on a retained cube $Q$ in case \textup{(c)}, since it is only
required to be locally a bi-Lipschitz homeomorphism near $Q$; the next
lemma shows that it is injective on all sufficiently small subcubes.

\begin{lemma}[Uniform charts on compact cubes]
\label{lem:compact-bilipschitz-charts}
Suppose that $f$ is locally a bi-Lipschitz homeomorphism onto an
open image on a neighborhood of a compact cube $Q$.  There are
$\rho>0$ and $\kappa\geq1$ such that every compact cube
$P\subseteq Q$ with $\ell(P)\leq\rho$ has an open neighborhood
on which $f$ is a $\kappa$-bi-Lipschitz homeomorphism onto an
open image.  Common constants can be chosen for any finite
family of such compact cubes.
\end{lemma}

\begin{proof}
Cover $Q$ by finitely many open chart domains and take the maximum
of their Lipschitz and inverse Lipschitz constants.
A Lebesgue number for this finite cover of $Q$, in the maximum
norm, gives $\rho>0$ such that every subset of $Q$ of diameter
at most $\rho$ lies in one chart.  A cube has maximum-norm
diameter equal to its side length.  For finitely many cubes,
take the minimum of the radii and the maximum of the constants.
\end{proof}

\begin{lemma}[Cubes in images and inverse images]
\label{lem:bilipschitz-cubes}
Let $P$ be a compact cube of side $L>0$, and let $g$ be a
$\kappa$-bi-Lipschitz homeomorphism on an open neighborhood of
$P$, with open image, where $\kappa\geq1$.
\begin{enumerate}[label=\textup{(\alph*)}]
\item The interior of $g(P)$ contains a compact cube of side
$L/(2\kappa)$.
\item If $W\subseteq g(P)$ is a compact cube of side $w>0$,
with center $y$, and $x=g^{-1}(y)$, the cube centered at $x$
of side $w/(4\kappa)$ lies in the interior of $P$ and is
mapped into $W$.
\end{enumerate}
\end{lemma}

\begin{proof}
Let $a$ be the center of $P$.  The homeomorphism property gives
$\partial g(P)=g(\partial P)$, and inverse Lipschitzness gives
\[
 \operatorname{dist}_\infty(g(a),\partial g(P))
       \geq L/(2\kappa).
\]
The cube centered at $g(a)$ of side $L/(2\kappa)$ therefore
lies in the interior of $g(P)$.

For (b), the distance from $y$ to $\partial g(P)$ is at least
$w/2$, because $W\subseteq g(P)$.  For every $b\in\partial P$,
\[
 \|x-b\|_\infty
 \geq\kappa^{-1}\|g(x)-g(b)\|_\infty
 \geq w/(2\kappa).
\]
Thus the proposed input cube lies in the interior of $P$.
For each of its points $u$,
$\|g(u)-y\|_\infty\leq\kappa\|u-x\|_\infty\leq w/8$,
so its image lies in $W$.
\end{proof}

The image restrictions and pullbacks in the following construction are
closely related to Richter's scalar-image coding
\cite[Sections~4.2 and~5]{Ric25}. Here we retain a compact family of
candidates of exact Hausdorff dimension $s$, while ensuring that the
prescribed real is recoverable from the jump of every image value
arising under \textup{(c)}.

\begin{proposition}[Simultaneous coded images]
\label{prop:general-cantor}
Let $k\geq1$, and let $(f_n)_{n\in\omega}$ be a countable family
of partial maps $\R^k\rightharpoonup\R^k$, with
$f_0=\mathrm{id}_{\R^k}$, each satisfying \textup{(QT)}.
For every $0<s<k$ and $c\in2^\omega$, there is a compact
perfect $K\subseteq[0,1]^k$ such that:
\begin{enumerate}[label=\textup{(\roman*)}]
\item $\dimH K=s$ and $K\subseteq E_k(s)$;
\item for every $x\in K$ and every $n$, one of \textup{(a)--(c)}
      holds for $f_n$ at $x$;
\item if $x\in K$ and \textup{(c)} holds for $f_n$ at $x$, then
\[
       f_n(x)\in E_k(s),
       \qquad c\leqT\Cut(f_n(x))'.
\]
\end{enumerate}
\end{proposition}

\begin{proof}
Choose witnesses $0<\varepsilon_n\leq1$, and put
\[
 \lambda=\frac{k}{s}>1,\quad
 N_j=2^{j^2},\quad
 w_j=2^{-\lceil\lambda N_j\rceil-3},\quad
 h_j=2^{j-N_j},\quad
 c_{<j}=\sum_{i<j}c(i)2^i.
\]
For $j\geq1$,
\begin{equation}\label{eq:euclidean-window-bounds}
       2^{-\lambda N_j-4}\leq w_j\leq2^{-N_j-4}.
\end{equation}
Let $e(j)$ be the exponent of $2$ dividing $j$.  Choose
$j_0\geq6$ with $2^{j_0}\geq\lambda$, and integers
$d_n\geq\varepsilon_n^{-1}$.  Introduce $f_n$ at stage
\[
       j_n=\min\{j\geq j_0:e(j)=n,\ j\geq d_n\}.
\]
The stages $j_n$ need not be computable; the decoding below uses only
the computable function $e$.

Use the lattice translate
\[
 G_j=\left\{2^{-N_j}v:
 \begin{array}{l}
 v\in\mathbb Z^k,\quad v_1\equiv c_{<j}\pmod{2^j},\\
 v_i\equiv0\pmod{2^j}\quad(2\leq i\leq k)
 \end{array}\right\},
\]
with spacing $h_j$ in each coordinate.
For $g\in G_j$, let $W_j(g)=g+[w_j,2w_j]^k$.

Construct finite families $\mathcal Q_j$ of compact cubes,
starting with $\mathcal Q_{j_0-1}=\{[0,1]^k\}$.
At stage $j\geq j_0$, put $n=e(j)$.
If $j=j_n$, in each parent $P$ choose by (QT) a retained cube
$\widetilde P\subseteq P$ with
$\ell(\widetilde P)\geq\varepsilon_n\ell(P)\geq\ell(P)/j$.
Record which of \textup{(a)--(c)} holds on its open neighborhood.
For the finitely many retained cubes in case \textup{(c)},
Lemma~\ref{lem:compact-bilipschitz-charts}
supplies common constants $\rho_n>0$ and $\kappa_n\geq1$.
If there are no such cubes, put $\rho_n=\kappa_n=1$.
At stages other than $j_n$, put $\widetilde P=P$.

Every descendant remains in its introduction cube.  Use
$g_{j,P}=f_n$ if $f_n$ has been introduced,
$\widetilde P$ lies in an introduction cube in case \textup{(c)},
$j\geq\kappa_n$, and $\ell(\widetilde P)\leq\rho_n$.
Otherwise use $g_{j,P}=\mathrm{id}$.
Then $g_{j,P}$ is a $j$-bi-Lipschitz homeomorphism
on an open neighborhood of $\widetilde P$, with open image,
and
\begin{equation}\label{eq:euclidean-working-map}
       \ell(\widetilde P)\geq\ell(P)/j.
\end{equation}
It is the condition $\ell(\widetilde P)\leq\rho_n$ that makes $f_n$
injective near $\widetilde P$; local invertibility alone would not
suffice on a large cube.

For $j\geq j_0-1$, put
\[
       \delta_j=\frac{w_j}{4j},
       \qquad\eta_j=\frac{h_j}{2j};
\]
for $j\geq j_0$, put
\[
       q_j=\left\lfloor
                  \frac{\delta_{j-1}}{4j^2h_j}
             \right\rfloor,\qquad m_j=q_j^k.
\]
All level-$j$ cubes will have side $\delta_j$.
Since $2^{2j-2}\geq\lambda$ and $2^{2j-2}\geq4j$,
\[
       N_j=2^{2j-1}N_{j-1}
          \geq(\lambda+4j)N_{j-1}.
\]
Equation~\eqref{eq:euclidean-window-bounds} gives
\begin{align}
 \log_2\frac{\delta_{j-1}}{j^2h_j}
 &\geq N_j-\lambda N_{j-1}-j-6-3\log_2j \notag\\
 &\geq3j-6-3\log_2j\geq3.
 \label{eq:euclidean-branch-numerics}
\end{align}
Therefore $q_j\geq2$ and
\begin{equation}\label{eq:euclidean-branch-bounds}
 \frac{\delta_{j-1}}{8j^2h_j}
       \leq q_j\leq
 \frac{\delta_{j-1}}{4j^2h_j}.
\end{equation}

Let $L=\ell(P)$ for a current parent.
Initially $L=1\geq\delta_{j_0-1}$; afterwards
$L=\delta_{j-1}$.
By Lemma~\ref{lem:bilipschitz-cubes}(a), the interior of
$g_{j,P}(\widetilde P)$ contains a compact cube $S_{j,P}$
of side $a=L/(2j^2)$.
Equation~\eqref{eq:euclidean-branch-numerics} gives $a\geq4h_j$.
In each coordinate, the admissible base points of windows
lying in $S_{j,P}$ range over an interval of length $a-w_j$
intersected with a translate of $h_j\mathbb Z$.
There are at least
\[
       \frac{a-w_j}{h_j}-1
       \geq\frac a{h_j}-2
       \geq\frac a{2h_j}
       \geq\frac{\delta_{j-1}}{4j^2h_j}
       \geq q_j
\]
such coordinate values.
Choose the $q_j$ smallest in each coordinate, giving exactly
$m_j=q_j^k$ windows.  For each selected window, let $y_g$
be its center and
$x_g=(g_{j,P}|\widetilde P)^{-1}(y_g)$.
Keep the cube centered at $x_g$ of side
$\delta_j=w_j/(4j)$.
By Lemma~\ref{lem:bilipschitz-cubes}(b), it lies in the
interior of $\widetilde P$ and maps into that window.
These are the children of $P$.

Distinct windows have maximum-norm distance at least
$h_j-w_j\geq h_j/2$.  The $j$-Lipschitz bound implies
distance at least $\eta_j=h_j/(2j)$ between children of
the same parent.  Children of distinct parents inherit
separation $\eta_{j-1}>\eta_j$.
Thus distinct cubes in $\mathcal Q_j$ have maximum-norm
distance at least $\eta_j$.
Each parent has exactly $m_j\geq2$ children, and
$\delta_j\to0$.  Consequently
\[
       K=\bigcap_{j\geq j_0}\bigcup\mathcal Q_j
\]
is nonempty, compact, and perfect.

The finitely many introduction cubes for $f_n$ cover $K$,
so, at every point of $K$, one of \textup{(a)--(c)} holds for $f_n$.
Fix $x\in K$ at which \textup{(c)} holds for $f_n$, and put
$y=f_n(x)$.  Its introduction cube is in case \textup{(c)}.
The sides of all subsequent parents tend to zero.
Hence at every sufficiently large stage $j$ with $e(j)=n$,
the parent of $x$ satisfies the activation conditions
$j\geq\kappa_n$ and $\ell(\widetilde P)\leq\rho_n$.
At each such stage the construction gives $v\in\mathbb Z^k$
with the prescribed residues and
\begin{equation}\label{eq:euclidean-coded-window}
       y-2^{-N_j}v\in[w_j,2w_j]^k.
\end{equation}
Thus
\[
       \|y-2^{-N_j}v\|_\infty
       \leq2w_j\leq2^{-\lceil kN_j/s\rceil},
\]
and $y\in E_k(s)$.
Moreover, because $0<w_j<2w_j<2^{-N_j}$,
\[
       \operatorname{cell}_{N_j}(y_1)\bmod2^j=c_{<j}.
\]
Each fixed bit is eventually correct along the computable
sequence $e(j)=n$, so $c\leqT\Cut(y)'$.
The identity gives $K\subseteq E_k(s)$, and hence
$\dimH K\leq s$ by Lemma~\ref{lem:approximation}.

For the reverse dimension bound, put $M_{j_0-1}=1$ and
$M_j=m_{j_0}\cdots m_j$.
Equal division of mass among children gives an atomless
Borel probability measure $\mu$ supported on $K$, with
$\mu(P)=1/M_j$ for $P\in\mathcal Q_j$.
Fix $0<t<s$.
By \eqref{eq:euclidean-branch-bounds},
\begin{align*}
 \log_2q_j
 &\geq N_j-\lambda N_{j-1}-j-9-3\log_2j\\
 &\geq N_j-\lambda N_{j-1}-4j,
\end{align*}
where $9+3\log_2j\leq3j$ for $j\geq6$.
For $j\geq j_0+1$, it follows that
\begin{equation}\label{eq:euclidean-mass-growth}
 \log_2M_{j-1}
 \geq k\log_2q_{j-1}
 \geq kN_{j-1}-k\lambda N_{j-2}-4k(j-1)
       =(k-o(1))N_{j-1}.
\end{equation}
Also,
\[
       \frac{\delta_{j-1}}{16j^3}
       \leq q_j\eta_j
       \leq\frac{\delta_{j-1}}{8j^3}<1,
\]
and therefore
\begin{equation}\label{eq:euclidean-mass-scale}
 0<-\log_2(q_j\eta_j)
 \leq\lambda N_{j-1}+4\log_2j+10
       =(\lambda+o(1))N_{j-1}.
\end{equation}
Since $t<s=k/\lambda$,
\[
 \log_2\bigl(M_{j-1}(q_j\eta_j)^t\bigr)
       \geq(k-\lambda t-o(1))N_{j-1}\longrightarrow\infty.
\]
Thus
\[
       B_t=\sup_{j\geq j_0}
       \bigl(M_{j-1}(q_j\eta_j)^t\bigr)^{-1}<\infty.
\]

Let $Q$ be a closed test cube of side $r$ with
$0<r<\eta_{j_0-1}$, and choose $j\geq j_0$ such that
$\eta_j\leq r<\eta_{j-1}$.
It meets at most one parent at level $j-1$.
It meets at most $(1+r/\eta_j)^k$ cubes at level $j$:
choose an intersection point from each cube and label it
by the integer coordinate vector obtained by dividing
its coordinates relative to the lower corner of $Q$
by $\eta_j$ and taking floors.
Two points with the same label have maximum-norm distance
less than $\eta_j$, so the labels are distinct.
There are at most $(1+\lfloor r/\eta_j\rfloor)^k$
labels, giving the asserted bound.
As $r\geq\eta_j$, this is at most $(2r/\eta_j)^k$.
Since $M_j=q_j^kM_{j-1}$,
\begin{align*}
 \mu(Q)
 &\leq\frac1{M_{j-1}}
       \min\left\{1,\left(\frac{2r}{q_j\eta_j}\right)^k\right\}\\
 &\leq\frac{2^t r^t}{M_{j-1}(q_j\eta_j)^t}
 \leq2^tB_t r^t.
\end{align*}
The middle inequality uses $0<t<k$ and
$\min\{1,a^k\}\leq a^t$ for $a\geq0$.
For $r\geq\eta_{j_0-1}$ use
$\mu(Q)\leq1\leq\eta_{j_0-1}^{-t}r^t$.
Lemma~\ref{lem:hausdorff}(c) now gives
$\mathcal H^t(K)>0$.
Since this holds for every $0<t<s$, it proves
$\dimH K\geq s$ and completes the proof.
\end{proof}

\begin{remark}\label{rem:general-cantor-decoding}
The decoding is uniform: a single Turing functional recovers $c$
from $\Cut(f_n(x))'$ and the index $n$, for every $x\in K$
at which \textup{(c)} holds for $f_n$. It reads only the first coordinate,
at the scales $N_j=2^{j^2}$ along the stages with $e(j)=n$, and so it
does not depend on $k$, on $s$, or on the chart data
$\rho_n,\kappa_n$.
\end{remark}

\begin{proposition}[Coded points and their reflections]
\label{prop:reflection-cantor}
For every integer $k\geq1$, every $d\in[0,k)$, every $z\in\R^k$,
and every $c\in2^\omega$, there is a compact set
$K=K_k(d,z,c)\subseteq[0,1]^k$ such that:
\begin{enumerate}[label=\textup{(\roman*)}]
\item $K$ is perfect;
\item $x\in U_k(d)$ and $z-x\in U_k(d)$ for every $x\in K$;
\item $c\leqT x$ and $c\leqT(z-x)$ for every $x\in K$;
\item $\dimH K=d$.
\end{enumerate}
In particular, $\dimH U_k(d)=d$ for every $d\in[0,k]$.
\end{proposition}

\begin{proof}
Fix $k,d,z,c$.  For $j\geq1$, put
\[
 \begin{split}
 N_j&=2^{j^2},\qquad h_j=2^{j-N_j},\qquad
 d_j=d+\frac{k-d}{j},\\
 a_j&=\left\lceil\frac{kN_j}{d_j}\right\rceil+3,
 \qquad w_j=2^{-a_j},\qquad
 c_{<j}=\sum_{i<j}c(i)2^i.
 \end{split}
\]
We have $d<d_j\leq k$, $d_j\geq k/j$, and $d_j\to d$.
Thus, including when $d=0$,
\begin{equation}\label{eq:reflection-exponents}
       N_j+3\leq a_j\leq jN_j+4.
\end{equation}
For $j\geq2$, the scales satisfy
\begin{equation}\label{eq:reflection-scales}
 w_{j-1}\geq8h_j,
 \qquad w_j\leq2^{-N_j-3}\leq h_j/4,
 \qquad w_j<w_{j-1},
 \qquad h_j<h_{j-1}.
\end{equation}
To check the first and third inequalities, use
$N_j=2^{2j-1}N_{j-1}$ and
\[
       N_j-(j-1)N_{j-1}\geq j+7 \quad(j\geq2).
\]
For $j=2$ the left side is $14$; for $j\geq3$ it is at least
$2^{2j-1}\geq j+7$.  Equation~\eqref{eq:reflection-exponents}
then gives
\[
 \log_2(w_{j-1}/h_j)
      =N_j-j-a_{j-1}
      \geq N_j-(j-1)N_{j-1}-j-4\geq3,
\]
and $a_j\geq N_j+3>(j-1)N_{j-1}+4\geq a_{j-1}$.
The other inequalities in \eqref{eq:reflection-scales} follow
directly from the definitions.

For $j\geq2$, let
\[
 G_j=\left\{2^{-N_j}v:
 \begin{array}{l}
  v\in\mathbb Z^k,\quad v_1\equiv c_{<j}\pmod{2^j},\\
  v_l\equiv0\pmod{2^j}\quad(2\leq l\leq k)
 \end{array}\right\}.
\]
This is a translate of $h_j\mathbb Z^k$.
For $g\in G_j$, let $W_j(g)=g+[w_j,2w_j]^k$.
Define the isometry
\[
       \varphi_j(x)=
       \begin{cases}
          x,&j\text{ even},\\
          z-x,&j\text{ odd}.
       \end{cases}
\]
Each $\varphi_j$ maps axis-parallel closed cubes onto
axis-parallel closed cubes of the same side length.
Finally, put
\[
       \kappa_j=\left\lfloor\frac{w_{j-1}}{2h_j}\right\rfloor,
       \qquad m_j=\kappa_j^k.
\]
By \eqref{eq:reflection-scales},
\begin{equation}\label{eq:reflection-branching}
       \kappa_j\geq4,
       \qquad
       \frac{w_{j-1}}{4h_j}\leq\kappa_j
                       \leq\frac{w_{j-1}}{2h_j}.
\end{equation}

We construct finite families $\mathcal Q_j$ of closed cubes,
starting with $\mathcal Q_1=\{[0,1]^k\}$.
Members of $\mathcal Q_j$ have side length $\ell_j$, where
$\ell_1=1$ and $\ell_j=w_j$ for $j\geq2$.
In particular, $\ell_{j-1}\geq w_{j-1}$.
Let $P\in\mathcal Q_{j-1}$, and write
\[
       \varphi_j(P)=\prod_{l=1}^k[\alpha_l,\alpha_l+\ell_{j-1}].
\]
A window $W_j(g)$ is contained in this cube exactly when
\[
       g_l\in[\alpha_l-w_j,\alpha_l+\ell_{j-1}-2w_j]
                  \quad(1\leq l\leq k).
\]
For each coordinate $l$, the admissible values form the
intersection of an interval of length $\ell_{j-1}-w_j$ with a
translate of $h_j\mathbb Z$.  Their number is at least
\[
 \frac{\ell_{j-1}-w_j}{h_j}-1
   \geq\frac{\ell_{j-1}}{h_j}-2
   \geq\frac{\ell_{j-1}}{2h_j}
   \geq\kappa_j.
\]
Here we used $w_j\leq h_j$ and
$\ell_{j-1}\geq w_{j-1}\geq8h_j$.
Choose the $\kappa_j$ smallest admissible values in each
coordinate and keep the cubes
\[
       \varphi_j^{-1}(W_j(g))
\]
for the resulting Cartesian product of coordinate choices.
These are exactly $m_j$ children of $P$, each of side $w_j$ and
contained in $P$.  Let $\mathcal Q_j$ be the collection of all
children and put
\[
       K=\bigcap_{j\geq1}\bigcup\mathcal Q_j.
\]

Two distinct windows at level $j$ differ in some base-point
coordinate by at least $h_j$, and therefore have
$\|\cdot\|_\infty$-distance at least
$h_j-w_j\geq h_j/2$.  Since $\varphi_j$ is an isometry, the same
holds for two children of a common parent.  Children of distinct
parents inherit distance at least $h_{j-1}/2\geq h_j/2$.
It follows inductively that distinct cubes in $\mathcal Q_j$
have $\|\cdot\|_\infty$-distance at least $h_j/2$.

In particular, each point of $K$ lies in a unique cube at every
level.  The set $K$ is nonempty and compact, every construction
cube has at least two children, and every child has descendants
at all later levels.  Since the diameters at level $j$ tend to
zero, no point of $K$ is isolated.  Thus $K$ is perfect, also
when $d=0$.

Let $x\in K$, let $\epsilon\in\{0,1\}$, and put $y=x$ if
$\epsilon=0$, and $y=z-x$ if $\epsilon=1$.  At every
$j\geq2$ with $j\equiv\epsilon\pmod2$, the construction gives
$g\in G_j$ such that
\begin{equation}\label{eq:reflection-window-membership}
       y-g\in[w_j,2w_j]^k.
\end{equation}
Fix a rational $t>d$.  Since $d_j\to d$, for all sufficiently
large such $j$ we have $d_j<t$, and
\[
 \|y-g\|_\infty
      \leq2w_j
      =2^{-\lceil kN_j/d_j\rceil-2}
      \leq2^{-\lceil kN_j/t\rceil}.
\]
There are infinitely many stages of the required parity, so
$y\in E_k(t)$.  Since this holds for every rational $t>d$, we
obtain $y\in U_k(d)$, proving assertion (ii).

To decode $c(i)$ from $\Cut(y)$, choose the least integer
$j\geq\max\{2,i+1\}$ with $j\equiv\epsilon\pmod2$.
Read the first coordinate cut and compute
$v=\operatorname{cell}_{N_j}(y_1)$.
By \eqref{eq:reflection-window-membership} and
$0<2w_j<2^{-N_j}$, this integer is the first coordinate index
of the corresponding $g=2^{-N_j}(v_1,\ldots,v_k)$.
Consequently
\[
        v\bmod2^j=c_{<j}.
\]
Its $i$th binary digit is $c(i)$.  This proves both reductions
in assertion (iii), without taking a Turing jump.

It remains to compute the dimension.
If $d=0$, assertion (ii) and Lemma~\ref{lem:approximation} already
give $\dimH K=0$.  Suppose that $d>0$ and fix $0<t<d$.
Put $M_1=1$ and $M_j=m_2\cdots m_j$ for $j\geq2$.
Giving every child an equal share of its parent's mass defines
a Borel probability measure $\mu$ supported on $K$, with
\[
       \mu(P)=1/M_j\qquad(P\in\mathcal Q_j).
\]
The separation of the cubes makes this identity unambiguous.
Also, $M_j\to\infty$, so $\mu$ has no atoms.

Equations \eqref{eq:reflection-exponents} and
\eqref{eq:reflection-branching} give, for $j\geq2$,
\begin{align*}
   \log_2\kappa_j
   &\geq N_j-j-a_{j-1}-2\\
   &\geq N_j-(j-1)N_{j-1}-j-6\\
   &\geq N_j-3jN_{j-1}.
\end{align*}
Since $M_j\geq m_j=\kappa_j^k$ and
$a_j\leq kN_j/d_j+4$, it follows that
\begin{equation}\label{eq:reflection-mass-growth}
 \log_2(M_jw_j^t)
    \geq kN_j\left(1-\frac{t}{d_j}-3j2^{1-2j}\right)-4t
       \longrightarrow\infty.
\end{equation}
Indeed, $t/d_j\to t/d<1$.  Thus
\[
       C_t=\sup_{j\geq1}(M_jw_j^t)^{-1}<\infty.
\]

Let $Q$ be any closed cube of side length $\ell>0$.
If $\ell\geq w_1$, then
\[
      \mu(Q)\leq1\leq(\ell/w_1)^t\leq C_t\ell^t,
\]
because $C_t\geq w_1^{-t}$.  Otherwise choose $j\geq2$ with
$w_j\leq\ell<w_{j-1}$.

First, $Q$ meets at most $2^k$ cubes of $\mathcal Q_{j-1}$.
For $j=2$ there is only one such cube.  For $j\geq3$, those cubes
have common side length $L=w_{j-1}>\ell$, and their centers have
pairwise $\|\cdot\|_\infty$-distance at least $L$.
The centers of cubes meeting $Q$ lie in a cube of side
$L+\ell<2L$.  Divide that cube into $2^k$ pieces, each of
$\|\cdot\|_\infty$-diameter less than $L$; every piece contains
at most one center.

Second, each parent $P\in\mathcal Q_{j-1}$ has at most
$(2+\ell/h_j)^k$ children meeting $Q$.
A child is a cube $u+[0,w_j]^k$, where
$u=g+w_j\mathbf1$ or $u=z-g-2w_j\mathbf1$ for $g\in G_j$.
Thus each coordinate of $u$ ranges over a translate of
$h_j\mathbb Z$.  If
$Q=\prod_l[\beta_l,\beta_l+\ell]$, such a child can meet $Q$
only if
\[
        u_l\in[\beta_l-w_j,\beta_l+\ell]
                       \quad(1\leq l\leq k).
\]
This interval has length $\ell+w_j\leq\ell+h_j$, so it contains
at most $2+\ell/h_j$ possible coordinate values.
Combining the two counts gives
\[
       \mu(Q)\leq\frac{2^k(2+\ell/h_j)^k}{M_j}.
\]
If $\ell<h_j$, then
\[
       \mu(Q)\leq6^k/M_j
                    \leq6^kC_tw_j^t
                    \leq6^kC_t\ell^t.
\]
If $\ell\geq h_j$, then $2+\ell/h_j\leq3\ell/h_j$.
Using $M_j=\kappa_j^kM_{j-1}$,
$\kappa_jh_j\geq w_{j-1}/4$, and $\ell<w_{j-1}$, we obtain
\begin{align*}
       \mu(Q)
       &\leq\frac{6^k\ell^k}{h_j^k\kappa_j^kM_{j-1}}\\
       &\leq\frac{24^k}{M_{j-1}}
                        \left(\frac\ell{w_{j-1}}\right)^k\\
       &\leq\frac{24^k}{M_{j-1}}
                        \left(\frac\ell{w_{j-1}}\right)^t
        \leq24^kC_t\ell^t.
\end{align*}
Thus $\mu(Q)\leq24^kC_t\ell^t$ for every closed cube.
Every subset of $\R^k$ of Euclidean diameter $\ell$ lies in
an axis-parallel cube of side $\ell$, and singleton sets have
zero $\mu$-measure.  The mass distribution principle therefore
gives $\mathcal H^t(K)\geq(24^kC_t)^{-1}>0$.
It follows that $\dimH K\geq d$.
The reverse inequality follows from $K\subseteq U_k(d)$ and
Lemma~\ref{lem:approximation}.  This proves assertion (iv).

Finally, the existence of $K\subseteq U_k(d)$ of dimension $d$
gives $\dimH U_k(d)=d$ for $d<k$.  For $d=k$, the equality follows
from $U_k(k)=\R^k$.
\end{proof}

\begin{remark}[Uniform reflection decoding]
\label{rem:reflection-decoding}
There are two Turing functionals $\Phi_0,\Phi_1$, independent of
$k,d,z,c$, such that, with the fixed vector coding,
\[
        \Phi_0(\Cut(x))=c,
        \qquad \Phi_1(\Cut(z-x))=c
                      \quad(x\in K_k(d,z,c)).
\]
The decoders use only the first coordinate, the fixed scales
$N_j=2^{j^2}$, and the parity of $j$.
The slowly decreasing exponents $d_j$ make $d=0$ part of the
same construction.  In particular, every $z\in\R^k$ can be
written as $z=x+(z-x)$ with both summands in $U_k(0)$, so
\[
          U_k(0)+U_k(0)=\R^k,
          \qquad\dimH U_k(0)=0.
\]
This observation will supply spanning pairs in the basis
construction.
\end{remark}

\section{Coded null sets}\label{sec:null-sets}

For \(k\ge1\), fix uniformly computable enumerations
\((\vartheta_i^k)_{i\in\omega}\) of \(\Q\cap(0,k)\) and
\((C_i^k)_{i\in\omega}\), where \(C_0^k=\varnothing\) and the other terms
list all open cubes with rational vertices and rational side length.
For a cube of side \(\ell\), put \(\sigma(C)=k\ell\), and let
\(\sigma(\varnothing)=0\). Thus \(\diam C\le\sigma(C)\).
For \(q\in\omega^\omega\), define
\[
 r_q^k=\vartheta_{q(0)}^k,\qquad
 t_q^k=\vartheta_{q(1)}^k,\qquad
 C_{q,n,i}^k=C_{q(2+\langle n,i\rangle)}^k
 \quad(n\ge1,\ i\in\omega).
\]

\begin{definition}\label{def:null-codes}
A code \(q\) is \(k\)-good if \(t_q^k<r_q^k\) and, for every \(n\ge1\)
and \(m\in\omega\),
\[
 \sigma(C_{q,n,m}^k)\le2^{-n},
 \qquad
 \sum_{i\le m}\sigma(C_{q,n,i}^k)^{t_q^k}\le2^{-n}.
\]
For a \(k\)-good code put
\[
 N_k(q)=\bigcap_{n\ge1}\bigcup_{i\in\omega}C_{q,n,i}^k,
 \qquad \theta_q^k=t_q^k.
\]
Otherwise put \(N_k(q)=\varnothing\) and \(\theta_q^k=0\).
\end{definition}

\begin{proposition}\label{prop:null-codes}
\begin{enumerate}[label=\textup{(\alph*)}]
\item Membership \(x\in N_k(q)\) is arithmetic, uniformly in \(k,q,x\),
      and every \(N_k(q)\) is \(G_\delta\).
\item For all \(k,q\), \(\theta_q^k<r_q^k\) and
      \(\Haus^t(N_k(q))=0\) whenever \(t>\theta_q^k\).
      In particular, \(\dimH N_k(q)<r_q^k\).
\item If \(r\in\Q\cap(0,k)\) and \(A\subseteq\R^k\) satisfies
      \(\dimH A<r\), then \(A\subseteq N_k(q)\) for a \(k\)-good code
      \(q\) with \(r_q^k=r\).
\end{enumerate}
\end{proposition}

\begin{proof}
The values \(\sigma(C)^t\), for rational \(t>0\), are uniformly computable
real algebraic numbers. Their comparisons with rational numbers are
arithmetic, so goodness is arithmetic.
Membership in a rational open cube is arithmetic in the cut of the point,
which proves (a).
For a good code the \(n\)-th cover has diameters at most \(2^{-n}\) and
total \(t_q^k\)-cost at most \(2^{-n}\).
Hence \(\Haus^{t_q^k}(N_k(q))=0\), and (b) follows from
Lemma~\ref{lem:hausdorff}. The nongood case is immediate.

For (c), choose a rational \(t\) with \(\dimH A<t<r\).
Fix \(n\ge1\), and let \(\delta=2^{-n}/(4k)\).
Choose a cover \(A\subseteq\bigcup_i V_i\) with
\[
 \diam V_i\le\delta,\qquad
 \sum_i(\diam V_i)^t<(4k)^{-t}2^{-n-1}.
\]
Put
\[
 \eta_i=\min\{\delta,((4k)^{-t}2^{-n-2-i})^{1/t}\};
 \qquad
 \sum_i\eta_i^t\le(4k)^{-t}2^{-n-1}.
\]
If \(V_i=\varnothing\), set \(D_i=\varnothing\).
Otherwise let \(\delta_i=\diam V_i\), choose \(v_i\in V_i\), a rational
point \(c_i\) with \(\|c_i-v_i\|_\infty<\eta_i/2\), and a rational
\(\ell_i\) satisfying
\[
                  2\delta_i+\eta_i<\ell_i\le2\delta_i+2\eta_i.
\]
Let \(D_i\) be the open cube of side \(\ell_i\) centered at \(c_i\).
Then \(V_i\subseteq D_i\), and
\[
                 \sigma(D_i)\le2k(\delta_i+\eta_i)\le2^{-n}.
\]
Using \((a+b)^t\le2^t(a^t+b^t)\), we obtain
\[
 \sum_i\sigma(D_i)^t
 \le(4k)^t\sum_i(\delta_i^t+\eta_i^t)<2^{-n}.
\]
Carry out this construction for every \(n\), and code the resulting cubes,
with the first two entries coding \(r\) and \(t\).
The resulting code is \(k\)-good and its null set contains \(A\).
\end{proof}

\begin{corollary}[Dimension test]\label{cor:dimension-test}
For \(B\subseteq\R^k\) and \(s\in[0,k]\),
\[
 \dimH B\ge s
 \quad\Longleftrightarrow\quad
 (\forall q\in\omega^\omega)\,
      (r_q^k<s\ \Longrightarrow\ B\nsubseteq N_k(q)).
\]
\end{corollary}

\begin{proof}
If \(B\subseteq N_k(q)\) and \(r_q^k<s\), then
\(\dimH B<r_q^k<s\).
Conversely, if \(\dimH B<s\), choose a rational
\(\dimH B<r<s\) and apply Proposition~\ref{prop:null-codes}(c).
The assertion for \(s=0\) is vacuous on the right and automatic on the left.
\end{proof}

\section{Proof of Theorem~\ref{thm:A}}\label{sec:general-closures}

Fix a coded function family as in Definition~\ref{def:function-family}.
We call the elements $\iota=(n,\gamma)\in I$ \emph{labels}, and write
$\iota^*=\widehat n\oplus\gamma$ for the real that codes $\iota$.
A single recursion of length $\omega_1$ handles all labels at once:
every record carries a label, all records belong to one history, and
a generator with label $\iota$ contributes only to the closures for
$\iota$. For a label $(n,\gamma)$, the rational tags
lie in $(0,n)$, and we use the approximation sets $E_n(r)$ of
Section~\ref{sec:cantor-sets} and the null-set codes $N_n(q)$ of
Section~\ref{sec:null-sets}.

For a finite tuple $\bar a$ of points of $\R^n$ we abbreviate
$\cl_\iota(\operatorname{ran}\bar a)$ to $\cl_\iota(\bar a)$, and
write $\cl_\iota(\bar a,x)$ for $\cl_\iota(\operatorname{ran}\bar a\cup\{x\})$.
Using projections, permutation and identification of variables, and
dummy variables, every element of $\cl_\iota(\bar a)$ can be written
as $f_{\iota,e}(\bar a)$ for a map $f_{\iota,e}$ whose arity is exactly
the length of $\bar a$.

\begin{lemma}\label{cl:effective-closure}
Let $\iota=(n,\gamma)\in I$.
\begin{enumerate}
 \item $\cl_\iota$ is extensive, monotone and idempotent, has finite
 character, and maps countable sets to countable sets. Moreover,
 $\Q^n\subseteq\cl_\iota(\varnothing)$.
 \item The relation $y\in\cl_\iota(\bar a)$, as a relation of $\iota$,
 the finite tuple $\bar a$ and $y$, is lightface $\Delta^1_1$. More
 generally, if a countable set $S\subseteq\R^n$ is enumerated by a
 finite or infinite sequence with code $\sigma$, then
 $y\in\cl_\iota(S)$ is $\Delta^1_1$ as a relation of $\iota^*$,
 $\sigma$ and $\Cut(y)$.
 \item Let $h(t)=f_{\iota,e}(\bar a,t)$. If $h$ is defined and
 injective on an open neighborhood of $x\in\R^n$ and $h(x)=y$, then
 \[
 \Cut(x)\leq_h\iota^*\oplus\Cut(\bar a)\oplus\Cut(y).
 \]
\end{enumerate}
\end{lemma}

\begin{proof}
The projections give extensiveness. Monotonicity and finite character
are clear from the definition of $\cl_\iota$, and idempotence follows
from closure under composition. The constants give
$\Q^n\subseteq\cl_\iota(\varnothing)$. Since $\mathcal F_\iota$ is
countable and a countable set has only countably many finite tuples,
$\cl_\iota$ maps countable sets to countable sets.

For~(2), $y\in\cl_\iota(\bar a)$ holds if and only if
$\operatorname{Eval}(n,\gamma,e,\bar a,y)$ holds for some $e$ such
that $f_{\iota,e}$ has the arity of $\bar a$. This is a number
quantifier applied to a $\Delta^1_1$ relation. For a sequence coded
by $\sigma$ we quantify, in addition, over finite sets of positions
in the sequence.

For~(3), choose a rational open cube $Q$ containing $x$ inside a
neighborhood on which $h$ is defined and injective. Then $\{\Cut(x)\}$
is the set of cuts of points $t\in Q$ with $f_{\iota,e}(\bar a,t)=y$,
which is a $\Delta^1_1(\iota^*,\Cut(\bar a),\Cut(y))$ singleton, since
$e$ and $Q$ are finite data. Its element is hyperarithmetic in these
parameters by Lemma~\ref{fact:effective}(c).
\end{proof}

We shall use the following simple observation. If a section
$h(t)=f_{\iota,e}(\bar a,t)$ is constant with value $v$ on a nonempty
open subset of $\R^n$, then
\begin{equation}\label{cl:constant-old}
 v\in\cl_\iota(\bar a).
\end{equation}
Indeed, choose a rational vector $q\in\Q^n$ in that open set.
Since $q\in\cl_\iota(\varnothing)$, the value
$v=f_{\iota,e}(\bar a,q)$ belongs to $\cl_\iota(\bar a)$ by
idempotence. This is where we use that the constants in (F1) include
every rational vector, so that they are dense in $\R^n$.

\begin{lemma}[Extension with protected points]\label{cl:extension}
Fix $\iota=(n,\gamma)\in I$.  Let $S,T\subseteq\R^n$ be countable
and suppose that
\[
 T\cap\cl_\iota(S)=\varnothing.
\]
Let $0<r<n$, let $N\subseteq\R^n$ have Hausdorff dimension less than
$r$, and let $d,u\in\omega^\omega$.  There is a point $z\in\R^n$
such that:
\begin{enumerate}
 \item $z\notin\cl_\iota(S)\cup N$ and
 $T\cap\cl_\iota(S\cup\{z\})=\varnothing$;
 \item for every finite tuple $\bar a$ from $S$ and every
 $y\in\cl_\iota(\bar a,z)\setminus\cl_\iota(\bar a)$,
 \[
 y\in E_n(r)
 \quad\hbox{and}\quad
 d\oplus\Cut(z)\leq_h\Cut(y);
 \]
 \item $u\leq_h\Cut(z)$.
\end{enumerate}
\end{lemma}

\begin{proof}
Choose a code $\sigma$ for an enumeration of $S$, and choose
$c\in2^\omega$ coding $\iota^*\oplus\sigma\oplus d\oplus u$.
The sections
\[
 h(t)=f_{\iota,e}(\bar a,t),
 \qquad \bar a\in S^{<\omega},\quad t\in\R^n,
\]
where $f_{\iota,e}$ has arity $|\bar a|+1$, form a countable family of
partial maps satisfying \textup{(QT)}, by (F3). We apply
Proposition~\ref{prop:general-cantor} to this family, with the identity
of $\R^n$ listed first, with $s=r$, and with the real $c$. It gives a
compact perfect set $K\subseteq[0,1]^n$ of Hausdorff dimension $r$
such that, for every section $h$ and every $x\in K$, one of
\textup{(a)--(c)} holds for $h$ at $x$, and such that, whenever \textup{(c)} holds
for $h$ at $x\in K$,
\begin{equation}\label{cl:cantor-image}
 h(x)\in E_n(r)
 \quad\hbox{and}\quad
 c\leq_h\Cut(h(x)).
\end{equation}
In particular, the identity section gives
\begin{equation}\label{cl:cantor-identity}
 c\leq_h\Cut(x)\qquad(x\in K).
\end{equation}

Put $M=\cl_\iota(S)$.  We claim that
\begin{equation}\label{cl:protected-bad}
 \{x\in K:
       T\cap\cl_\iota(S\cup\{x\})\ne\varnothing\}
 \quad\hbox{is countable}.
\end{equation}
Indeed, if $w\in T\cap\cl_\iota(S\cup\{x\})$, then $w=h(x)$ for one of
the sections $h$ above, and there are only countably many pairs
$(w,h)$. So fix $w\in T$ and a section $h(t)=f_{\iota,e}(\bar a,t)$,
and let $x\in K$ satisfy $h(x)=w$. Then $h$ is defined at $x$, and it
is not constant near $x$, for otherwise~\eqref{cl:constant-old} would
give $w\in\cl_\iota(\bar a)\subseteq M$, contrary to $T\cap
M=\varnothing$. Thus \textup{(a)} and \textup{(b)} are excluded at $x$,
so \textup{(c)} holds. Hence $h$ is injective on a neighborhood of $x$,
and $x$ is an isolated point of
$\{x'\in K:h(x')=w\}$. This set is therefore countable, which
proves~\eqref{cl:protected-bad}.

Both $M$ and the set in~\eqref{cl:protected-bad} are countable, so
their union with $N$ has Hausdorff dimension less than
$r=\dim_{\mathrm H}K$. Choose $z\in K$ outside this union. Then (1)
holds, and (3) follows from~\eqref{cl:cantor-identity}.

For~(2), fix $\bar a$ and $y$, and write $y=h(z)$ with
$h(t)=f_{\iota,e}(\bar a,t)$. Since $y\notin\cl_\iota(\bar a)$,
\eqref{cl:constant-old} shows that $h$ is not constant near $z$. Thus
\textup{(c)} holds for $h$ at $z$, and~\eqref{cl:cantor-image}
gives $y\in E_n(r)$ and $c\leq_h\Cut(y)$. In particular $d$, $\iota^*$
and $\Cut(\bar a)$ are hyperarithmetic in $\Cut(y)$. Since $h$ is
injective near $z$, Lemma~\ref{cl:effective-closure}(3) yields
\[
 \Cut(z)\leq_h\iota^*\oplus\Cut(\bar a)\oplus\Cut(y)
           \leq_h\Cut(y),
\]
which completes the proof of~(2).
\end{proof}

Fix an effective coding $r\mapsto\ulcorner r\urcorner$ of the positive
rationals. Using the join codes of Section~\ref{sec:preliminaries}, we
distinguish the following kinds of reals.
\begin{enumerate}
 \item A \emph{null record} has the form $0^\frown v$, with $v$ arbitrary.
 \item A \emph{generator record} has the form
 \[
 b=1^\frown\bigl(\widehat n\oplus\gamma
        \oplus\widehat{\ulcorner r\urcorner}\oplus\Cut(z)\bigr),
 \]
 where $\iota=(n,\gamma)\in I$, $r\in\Q\cap(0,n)$, and $z\in\R^n$.
 Write $\operatorname{lab}(b)=\iota$,
 $\operatorname{tag}(b)=r$, and $\operatorname{pt}(b)=z$.
 \item A \emph{reservation record} has the form
 \[
 b=2^\frown\bigl(\widehat n\oplus\gamma\oplus\Cut(t)\bigr),
 \qquad (n,\gamma)\in I,\quad t\in\R^n.
 \]
 Its label is $(n,\gamma)$ and its point is $t$.
\end{enumerate}
Let $\operatorname{Gen}$ and $\operatorname{Res}$ denote the sets of
generator and reservation records. They are $\Delta^1_1$, since $I$ is
$\Delta^1_1$ and, once $n$ has been read off, the condition $r<n$ is
decidable and the validity of the vector cut is arithmetic by
Lemma~\ref{lem:cuts}. The label, tag and point of a
record are computable from it, and conversely a record is computable
from its label, its tag (if any) and the cut of its point. Reservation
points are never used as generators; they are the points to be kept
out of the closures.

A \emph{dimension demand} has the canonical form
\[
 p=1^\frown\bigl(\widehat n\oplus\gamma\oplus q\bigr),
 \qquad (n,\gamma)\in I,\quad q\in\omega^\omega.
\]
Its label is $\iota=(n,\gamma)$, and it uses $r_q^n$ and $N_n(q)$.
A \emph{reservation demand} has the form
\[
 p=2^\frown\bigl(\widehat n\oplus\gamma\bigr),
 \qquad(n,\gamma)\in I.
\]
Both kinds of demands form $\Delta^1_1$ sets. Each label has one
reservation demand, and one dimension demand for every code $q$.

For a history $A$ and a label $\iota$, define
\begin{align*}
 S_\iota(A)&=
 \{\operatorname{pt}(b):b\in\operatorname{ran}A\cap\operatorname{Gen},
                         \operatorname{lab}(b)=\iota\},\\
 T_\iota(A)&=
 \{\operatorname{pt}(b):b\in\operatorname{ran}A\cap\operatorname{Res},
                         \operatorname{lab}(b)=\iota\},\\
 M_\iota(A)&=\cl_\iota(S_\iota(A)).
\end{align*}
These are countable subsets of $\R^n$, where $\iota=(n,\gamma)$.
Call $A$ \emph{coherent} if all its entries are null, generator, or
reservation records, and
\begin{equation}\label{cl:coherence}
 T_\iota(A)\cap M_\iota(A)=\varnothing
 \qquad\hbox{for every }\iota\in I.
\end{equation}

\begin{lemma}\label{cl:coherence-effective}
Coherence is a lightface $\Delta^1_1$ property of histories.
\end{lemma}

\begin{proof}
That every entry of $A$ is a record is a number quantifier applied to
a $\Delta^1_1$ condition. By finite character, \eqref{cl:coherence}
says that for every position of a reservation record in $A$ and every
finite set of positions of generator records with the same label, the
reserved point is not in the closure of these generator points. The
labels are read off from the records, so only number quantifiers
occur, and the matrix is $\Delta^1_1$ by
Lemma~\ref{cl:effective-closure}(2).
\end{proof}

Let $\mathcal P\subseteq
(\omega^\omega)^{\leq\omega}\times\omega^\omega\times\omega^\omega$
be the relation defined by the following alternatives.
\begin{enumerate}
 \item[(P0)] $A$ is not coherent, and $b$ is arbitrary.
 \item[(P1)] $A$ is coherent, $p$ is not a demand, and $b$ is a null record.
 \item[(P2)] $A$ is coherent, $p$ is the reservation demand with label
 $\iota=(n,\gamma)$, and $b$ is a reservation record with that label
 and point $t\notin M_\iota(A)$.
 \item[(P3)] $A$ is coherent, $p$ is the dimension demand with label
 $\iota=(n,\gamma)$ and code $q$, and $b$ is a generator record with
 that label, tag $r_q^n$, and point $z\in\R^n$, satisfying:
 \begin{enumerate}
  \item $z\notin M_\iota(A)\cup N_n(q)$;
  \item $T_\iota(A)\cap
       \cl_\iota(S_\iota(A)\cup\{z\})=\varnothing$;
  \item for every finite tuple $\bar a$ from $S_\iota(A)$ and every
  $y\in\R^n$,
  \begin{equation}\label{cl:selection-output}
  \begin{split}
   y\in\cl_\iota(\bar a,z)\setminus\cl_\iota(\bar a)
   \quad\Longrightarrow\quad
   &y\in E_n(r_q^n)\ \text{ and }\\[-2pt]
   &A\oplus p\oplus\Cut(z)\leq_h\Cut(y).
  \end{split}
  \end{equation}
 \end{enumerate}
\end{enumerate}
Note that $A$ in~\eqref{cl:selection-output} is the whole history,
including the records with other labels.

\begin{lemma}\label{cl:selection-lemma}
The relation $\mathcal P$ is lightface $\Pi^1_1$, and every section
\[
 \mathcal P_{A,p}=\{b:(A,p,b)\in\mathcal P\}
\]
is cofinal in the hyperarithmetic degrees.  If $A$ is coherent and
$(A,p,b)\in\mathcal P$, then any history enumerating
$\operatorname{ran}A\cup\{b\}$ is coherent.
\end{lemma}

\begin{proof}
Records, demands and coherence are $\Delta^1_1$, and so is membership
in the closure of the points at given positions of $A$, by
Lemma~\ref{cl:effective-closure}(2). Together with
Proposition~\ref{prop:null-codes}, this shows that (P2), (P3)(a) and
(P3)(b) are $\Delta^1_1$. In (P3)(c) the hypothesis of the implication
is $\Delta^1_1$, and its conclusion is $\Pi^1_1$, because membership in
$E_n(r)$ is arithmetic and $\leq_h$ is $\Pi^1_1$. The tuple $\bar a$
is given by finitely many positions in $A$, and the quantifier over
$y$ is a universal real quantifier, so (P3)(c) is $\Pi^1_1$. Hence
$\mathcal P$ is $\Pi^1_1$.

For cofinality, fix $A,p$ and a real $u$. If $A$ is not coherent, take
$b=u$; if $A$ is coherent and $p$ is not a demand, take $b=0^\frown u$.
For a reservation demand with label $\iota=(n,\gamma)$, apply
Proposition~\ref{prop:general-cantor} to the identity of $\R^n$ alone,
with $s=n/2$ and a real $c\in2^\omega$ computing $u$. Every point of
the resulting set computes $u$ in the jump of its cut, and since the
set is uncountable we may choose such a point $t\notin M_\iota(A)$.
Its reservation record $b$ satisfies (P2), and
$u\leq_h\Cut(t)\leq_T b$.

For a dimension demand with label $\iota=(n,\gamma)$ and code $q$,
apply Lemma~\ref{cl:extension} with
\[
 S=S_\iota(A),\qquad T=T_\iota(A),\qquad
 r=r_q^n,\qquad N=N_n(q),\qquad d=A\oplus p.
\]
Its hypotheses hold: $T\cap\cl_\iota(S)=\varnothing$ by coherence, and
$\dim_{\mathrm H}N_n(q)<r_q^n<n$ by Proposition~\ref{prop:null-codes}.
The generator record $b$ of the resulting point $z$, with tag $r_q^n$,
satisfies (P3), and $u\leq_h\Cut(z)\leq_T b$.

Finally, let $A$ be coherent and $(A,p,b)\in\mathcal P$. A null record
changes none of the sets $S_\iota$ and $T_\iota$. Under (P2) a point
outside $M_\iota(A)$ is added to $T_\iota(A)$, and under (P3) the
new generator keeps the closure disjoint from $T_\iota(A)$, by
(P3)(b). The other labels are not affected, so coherence is preserved.
\end{proof}

Applying Theorem~\ref{fact:selection} to $\mathcal P$, we obtain a
lightface $\Pi^1_1$ set $Y$ compatible with $\mathcal P$, witnessed by
\[
 \omega^\omega=\{p_\xi:\xi<\omega_1\},\qquad
 Y=\{b_\xi:\xi<\omega_1\},
\]
and histories $A_\xi$ with
\[
 \operatorname{ran}A_\xi=\{b_\eta:\eta<\xi\},
 \qquad(A_\xi,p_\xi,b_\xi)\in\mathcal P.
\]
We show by induction on $\xi$ that every $A_\xi$ is coherent. Suppose
that $A_\eta$ is coherent for all $\eta<\xi$. Then each $b_\eta$,
$\eta<\xi$, is a record, and any history with range
$\operatorname{ran}A_\eta\cup\{b_\eta\}$ is coherent by
Lemma~\ref{cl:selection-lemma}. If~\eqref{cl:coherence} failed for
$A_\xi$, the failure would be witnessed, by finite character, by one
reservation record and finitely many generator records of $A_\xi$;
if $\eta$ is the largest of their stages, they all lie in
$\operatorname{ran}A_\eta\cup\{b_\eta\}$, a contradiction. Thus
alternative (P0) never occurs: every dimension demand is answered by
a generator record and every reservation demand by a reservation
record.

For $\iota=(n,\gamma)\in I$ and $\alpha\in[0,n]$, define
\begin{align*}
 S_\iota(\alpha)&=
 \{\operatorname{pt}(b):b\in Y\cap\operatorname{Gen},
       \operatorname{lab}(b)=\iota,
       \operatorname{tag}(b)<\alpha\},\\
 A_\iota(\alpha)&=\cl_\iota(S_\iota(\alpha)).
\end{align*}
Clearly each $A_\iota(\alpha)$ is closed under $\mathcal F_\iota$,
$A_\iota(\alpha)\subseteq A_\iota(\beta)$ for $\alpha\leq\beta$, and
$A_\iota(0)=\cl_\iota(\varnothing)$, since all tags are positive.

\begin{lemma}[Finite support and recovery]\label{cl:finite-support}
Fix $\iota=(n,\gamma)\in I$.  If $0<\alpha\leq n$ and
$y\in A_\iota(\alpha)\setminus\cl_\iota(\varnothing)$, then
there are generator records $b_{\xi_1},\ldots,b_{\xi_m}$, with
$\xi_1<\cdots<\xi_m$ and $m\geq1$, such that, writing
$z_i=\operatorname{pt}(b_{\xi_i})$ and
$r_i=\operatorname{tag}(b_{\xi_i})$,
\begin{enumerate}
 \item their labels are $\iota$ and every $r_i<\alpha$;
 \item $y\in\cl_\iota(z_1,\ldots,z_m)$;
 \item $y\in E_n(r_m)$;
 \item $b_{\xi_i}\leq_h\Cut(y)$ for $1\leq i\leq m$.
\end{enumerate}
\end{lemma}

\begin{proof}
By finite character, $y$ lies in the closure of finitely many points of
$S_\iota(\alpha)$. Choose such a set of minimal size, and list the
corresponding records in the order of their stages. This gives (1)
and (2), and $m\geq1$ since $y\notin\cl_\iota(\varnothing)$. Put
$\xi=\xi_m$ and $\bar a=(z_1,\ldots,z_{m-1})$. By minimality,
\[
 y\in\cl_\iota(\bar a,z_m)\setminus\cl_\iota(\bar a),
\]
and $\bar a$ is a tuple from $S_\iota(A_\xi)$. Condition (P3)(c) at
stage $\xi$ therefore gives (3), together with
\[
 A_\xi\oplus p_\xi\oplus\Cut(z_m)\leq_h\Cut(y).
\]
For $i<m$, the record $b_{\xi_i}$ is an entry of $A_\xi$, hence
hyperarithmetic in $\Cut(y)$. The record $b_\xi$ is computable from
$p_\xi$ and $\Cut(z_m)$, as the demand $p_\xi$ determines the label and
the tag. This proves (4).
\end{proof}

We stress that $y$ may well lie in the closure of all the generators
chosen before stage $\xi_m$; the argument only uses the generators in
a minimal support of $y$, and this is why the output condition in
(P3)(c) is imposed over every finite tuple $\bar a$. Neither this
argument nor the rest of the proof uses an exchange property of
$\cl_\iota$.

\begin{proof}[Proof of Theorem~\ref{thm:A}]
Fix a label $\iota=(n,\gamma)\in I$. We first verify (1)--(3) for the
sets $A_\iota(\alpha)$ defined above, and then show that they are the
sections of a single lightface $\Pi^1_1$ set $G$.

Let $0<\alpha\leq n$. By Lemma~\ref{cl:finite-support}(3),
\[
 A_\iota(\alpha)\subseteq
 \cl_\iota(\varnothing)\ \cup\!
 \bigcup_{\substack{r\in\Q\\0<r<\alpha}}E_n(r),
\]
and since $\cl_\iota(\varnothing)$ is countable,
$\dim_{\mathrm H}A_\iota(\alpha)\leq\alpha$ by
Lemmas~\ref{lem:hausdorff}(b) and~\ref{lem:approximation}. Suppose
that $\dim_{\mathrm H}A_\iota(\alpha)<\alpha$, and choose a rational
$r$ with $\dim_{\mathrm H}A_\iota(\alpha)<r<\alpha$; note that $r<n$
even if $\alpha=n$. By Proposition~\ref{prop:null-codes}(c) there is a
code $q$ with $r_q^n=r$ and $A_\iota(\alpha)\subseteq N_n(q)$. The
dimension demand for $\iota$ and $q$ is answered by a generator of tag
$r<\alpha$, which therefore lies in $A_\iota(\alpha)$; but by (P3)(a)
it lies outside $N_n(q)$, a contradiction. Hence
$\dim_{\mathrm H}A_\iota(\alpha)=\alpha$. For $\alpha=0$ the set
$A_\iota(0)=\cl_\iota(\varnothing)$ is countable.

The reservation demand for $\iota$ is answered by a record with some
point $t_\iota\in\R^n$. If $t_\iota$ were in $A_\iota(n)$, it would be
in the closure of finitely many generators with label $\iota$. All of
their records, and the reservation record, appear in $A_\xi$ for some
$\xi<\omega_1$, and $A_\xi$ would not be coherent. Hence
$t_\iota\notin A_\iota(n)$, and every
$A_\iota(\alpha)$ is proper. (The reservation demand may well come
after some generators with label $\iota$ have been chosen; what
matters is that the closure at its stage is countable, and that
coherence protects $t_\iota$ from then on.)

The inclusions $A_\iota(\alpha)\subseteq A_\iota(\beta)$ for
$\alpha<\beta$ are strict because the two sets have different
dimensions. Let $\alpha>0$ and $y\in A_\iota(\alpha)$, and let $F$ be
a finite set of generators of tag less than $\alpha$ with
$y\in\cl_\iota(F)$. Choose a rational $r$ with $0<r<\alpha$ that
exceeds the tags of all points of $F$. Then $y\in A_\iota(r)$, and
together with monotonicity this proves
\[
 A_\iota(\alpha)=
 \bigcup_{\substack{r\in\Q\\0<r<\alpha}}A_\iota(r).
\]

It remains to verify joint definability.
For $(n,\gamma)\in I$, $\alpha\in[0,n]$ and $y\in\R^n$, put
\[
 v=\widehat n\oplus\gamma\oplus\Cut(\alpha)\oplus\Cut(y),
\]
and let $W(v,t)$ say that $t$ codes a nonempty finite list
$(c_1,\ldots,c_m)$ of generator records such that
\[
 c_i\in Y,\qquad
 \operatorname{lab}(c_i)=(n,\gamma),\qquad
 \operatorname{tag}(c_i)<\alpha
 \quad(1\leq i\leq m),
\]
and
\[
 y\in\cl_{(n,\gamma)}
        (\operatorname{pt}(c_1),\ldots,\operatorname{pt}(c_m)).
\]
The relation $W$ is lightface $\Pi^1_1$: membership in $Y$ is
$\Pi^1_1$, the remaining clauses are $\Delta^1_1$ by
Lemma~\ref{cl:effective-closure}(2), and only number quantifiers
over the entries of the list are involved. We claim that, for
$\iota=(n,\gamma)$,
\begin{equation}\label{cl:sg-definition}
 \begin{split}
 y\in A_\iota(\alpha)\quad\Longleftrightarrow\quad
 &y\in\cl_\iota(\varnothing)\ \text{ or }\\[-2pt]
 & (\exists t\leq_h v)\,W(v,t).
 \end{split}
\end{equation}
The implication from right to left is clear. Conversely, if
$y\in A_\iota(\alpha)\setminus\cl_\iota(\varnothing)$, then
$\alpha>0$, and the records given by Lemma~\ref{cl:finite-support} are
hyperarithmetic in $\Cut(y)$, hence in $v$; their list witnesses
$W(v,t)$ for some $t\leq_h v$. Since $y\in\cl_\iota(\varnothing)$ is
a $\Delta^1_1$ condition, the Spector--Gandy theorem,
Lemma~\ref{fact:effective}(b), shows that
\[
 G=\{(n,\gamma,\alpha,y):(n,\gamma)\in I,\
          0\leq\alpha\leq n,\ y\in\R^n,\ y\in A_{(n,\gamma)}(\alpha)\}
\]
is lightface $\Pi^1_1$. Its sections $A_\iota(\alpha)$ are then
$\Pi^1_1(\gamma,\alpha)$.

If the family is given relative to a real $\delta$, the same argument
works with $\delta$ added as a parameter throughout: we use the
selection theorem relative to $\delta$, and in the proof of
Lemma~\ref{cl:extension} we let $c$ compute $\delta$ as well. This
gives $G\in\Pi^1_1(\delta)$.
\end{proof}

For the field applications, take \(n=1\) and the single label
\(I=\{(1,\gamma_0)\}\), where \(\gamma_0\) is the constant zero
sequence, and omit the label from the notation; the following
observation will verify (F3).

\begin{lemma}\label{lem:one-variable}
Let \(h\) be a partial map \(\R\rightharpoonup\R\), and let
\(D\subseteq\R\) be a finite set such that on each component of
\(\R\setminus D\) the map \(h\) is undefined, or constant, or \(C^1\)
with nowhere vanishing derivative. Then \(h\) satisfies \textup{(F3)}
with \(\varepsilon_h=1/(2|D|+2)\).
\end{lemma}

\begin{proof}
On an open interval where \(h\) is \(C^1\) with nowhere vanishing
derivative, \(h\) is strictly monotone, and near each point
\(|h'|\) is bounded above and away from zero; hence \(h\) is a local
bi-Lipschitz homeomorphism there. Thus on each component of
\(\R\setminus D\), and on every open subset of that component,
one of \textup{(a)--(c)} holds for \(h\). Put \(d=|D|\).
A compact interval \(Q\) is split by \(D\)
into at most \(d+1\) pieces, so one of them has length at least
\(\ell(Q)/(d+1)\). The closed middle half \(Q'\) of this piece
satisfies \(\ell(Q')\geq\ell(Q)/(2d+2)\), and the interior of the
piece is an open neighborhood of \(Q'\) on which one of
\textup{(a)--(c)} holds for \(h\).
\end{proof}

Field terms are built from variables \(x_1,x_2,\ldots\) and rational
constants by means of \(+\), \(-\), \(\cdot\) and the inverse
\(s\mapsto s^{-1}\). Fix a
computable enumeration \((t_e,m_e)_{e\in\omega}\) of the pairs
consisting of a field term and a number \(m\) such that the variables
of the term are among \(x_1,\ldots,x_m\). For \(\bar a\in\R^{m_e}\), the
value \(t_e(\bar a)\) is defined by recursion on the term, with
\(s^{-1}\) defined exactly when \(s\) is defined and nonzero; let
\(f_e\) be the partial map \(\bar a\mapsto t_e(\bar a)\) of arity
\(m_e\).

By recursion on a term \(t\) in \(x_1,\ldots,x_m\), we compute rational
polynomials \(P_t,Q_t\) and a finite list \(R_{t,1},\ldots,R_{t,k}\) of
rational polynomials in \(x_1,\ldots,x_m\) such that, for every
\(\bar a\in\R^m\), \(t(\bar a)\) is defined if and only if
\(R_{t,i}(\bar a)\ne0\) for all \(i\), and in that case
\(Q_t(\bar a)\ne0\) and \(t(\bar a)=P_t(\bar a)/Q_t(\bar a)\). For
variables and constants this is clear, with an empty list. For
\(s\pm s'\) and \(s\cdot s'\) we join the two lists and use the usual
formulas for fractions; for \(s^{-1}\) we add \(P_s\) to the list of
\(s\) and put \(P_{s^{-1}}=Q_s\), \(Q_{s^{-1}}=P_s\).
Thus \(t\) defines exactly the same partial map as the fraction
\(P_t\prod_iR_{t,i}/(Q_t\prod_iR_{t,i})\). Conversely, every fraction
of polynomials with rational coefficients is a field term, so this is
precisely the family used in the Introduction.

\begin{proof}[Proof of Corollary~\ref{cor:subfields}]
We check the conditions of Definition~\ref{def:function-family} for
\(\mathcal F=\{f_e:e\in\omega\}\). The family contains the projections
and the rational constants. It is closed under permutation and
identification of variables and under dummy variables, since these
amount to substitutions of variables in the terms. It is closed under
composition, since
\(\bar a\mapsto t(s_1(\bar a),\ldots,s_k(\bar a))\) is given by the
term \(t(s_1,\ldots,s_k)+0\cdot(s_1+\cdots+s_k)\); the second summand
makes it undefined wherever one of the \(s_i\) is. This proves (F1).
The arity \(m_e\) is computable from \(e\), and
\[
 f_e(\bar a)=y
 \quad\Longleftrightarrow\quad
 \bigwedge_{i}R_{t_e,i}(\bar a)\ne0
 \quad\hbox{and}\quad
 Q_{t_e}(\bar a)\,y-P_{t_e}(\bar a)=0
\]
is arithmetic, uniformly in \(e\), \(\bar a\) and \(y\), by
Lemma~\ref{lem:cuts}(c). This proves (F2).

For (F3), let \(h\) be a unary section of some \(f_e\); after a
permutation of variables, \(h(x)=t(\bar a,x)\) for a field term \(t\)
and a real tuple \(\bar a\). Put \(p(x)=P_t(\bar a,x)\),
\(q(x)=Q_t(\bar a,x)\) and \(r_i(x)=R_{t,i}(\bar a,x)\). If some
\(r_i\) is the zero polynomial, \(h\) is nowhere defined. Otherwise
let \(D_0\) be the finite set of zeros of the \(r_i\); on
\(\R\setminus D_0\) the map \(h\) is defined and equal to \(p/q\). If
the polynomial \(p'q-pq'\) vanishes identically, \(h\) is constant on
each component of \(\R\setminus D_0\). If not, then on each component
of the complement of \(D_0\) and of the zeros of \(p'q-pq'\), \(h\) is
\(C^1\) with nowhere vanishing derivative \((p'q-pq')/q^2\). In either
case Lemma~\ref{lem:one-variable} gives (F3).

For \(A\subseteq\R\), \(\cl(A)\) is the set of values of field terms at
tuples from \(A\), that is, the subfield \(\Q(A)\) generated by \(A\).
Hence the sets closed under \(\mathcal F\) are exactly the subfields of
\(\R\), and \(\cl(\varnothing)=\Q\). Theorem~\ref{thm:A} gives a
lightface \(\Pi^1_1\) set, and we let \(G_{\mathrm{fld}}\) be its section
at the computable label \((1,\gamma_0)\), which is again lightface
\(\Pi^1_1\). Its sections \(F(\alpha)\) are the sets \(A(\alpha)\) of
the theorem and have the required properties.
\end{proof}

\begin{proof}[Proof of Corollary~\ref{cor:exponential}]
We apply Theorem~\ref{thm:A} with \(n=1\) to the family
\(\mathcal F\) of all \(\varnothing\)-definable partial functions
\(\R^m\rightharpoonup\R\), \(m\geq0\), of \(\R_{\exp}\).
This family is countable and satisfies (F1).
Wilkie's theorem \cite{Wil96} gives model completeness and
o-minimality of \(\R_{\exp}\); we use these to verify (F2) and (F3),
respectively.

For (F2), fix a computable enumeration
\((m_e,k_e,\tau_e)_{e\in\omega}\), where \(\tau_e(\bar x,y,\bar z)\)
is a term in the ordered-ring language with \(\exp\),
\(|\bar x|=m_e\) and \(|\bar z|=k_e\). Put
\[
 R_e(\bar x,y)\quad\Longleftrightarrow\quad
 \exists\bar z\in\R^{k_e}\ \tau_e(\bar x,y,\bar z)=0,
\]
and let \(f_e(\bar x)\) be the unique \(y\) satisfying \(R_e(\bar x,y)\)
when there is exactly one such \(y\), and undefined otherwise.
By model completeness, every \(\varnothing\)-definable graph has
an existential definition. The usual effective encoding of
inequalities and Boolean combinations by equations with auxiliary
variables puts each such definition in the displayed form.
Thus \(\mathcal F=\{f_e:e\in\omega\}\), with computable arities.

Truth of existential formulas is arithmetic, uniformly in the formula
and the cuts of its real inputs. Indeed, each term is a uniformly
computable continuous function, and compactness gives
\[
 \exists\bar z\in\R^k\ \tau(\bar a,\bar z)=0
 \quad\Longleftrightarrow\quad
 \exists N\geq1\ \forall j\geq1\
 \exists\bar q\in(\Q\cap[-N,N])^k\
 |\tau(\bar a,\bar q)|<2^{-j}.
\]
The right-hand side has only number quantifiers, and strict
inequalities for computable continuous terms are semidecidable from
the input cuts. The same observation applies after the effective
equational encoding of any existential formula. Consequently
\[
 f_e(\bar a)=y
 \quad\Longleftrightarrow\quad
 R_e(\bar a,y)\ \land\
 \neg\exists v\,\bigl(v\ne y\land R_e(\bar a,v)\bigr)
\]
is arithmetic uniformly in \(e,\bar a,y\), proving (F2).
This enumeration does not require an effective procedure for finding
an existential equivalent of an arbitrary formula.

For (F3), let \(h:\R\rightharpoonup\R\) be any partial function
definable, with parameters, in an o-minimal expansion of the real
field. Its domain is a finite union of points and intervals.
The standard piecewise \(C^1\) theorem
\cite[Chapter~7, \S3]{vdD98} gives a finite partition of \(\R\) into
points and open intervals such that, on each open interval, \(h\)
is either undefined or \(C^1\).
On each interval of the latter kind, \(h'\) is definable, so its zero
set is a finite union of points and intervals. Refining the partition
at the endpoints and isolated points of these zero sets makes \(h'\)
either identically zero or nowhere zero on each remaining open
interval. In the former case \(h\) is constant by the mean value
theorem. Lemma~\ref{lem:one-variable} therefore gives (F3).
Fixing the other arguments of a map in \(\mathcal F\) preserves
definability with parameters, so this applies to every required section.

The \(\mathcal F\)-closure of \(A\subseteq\R\) is precisely
\(\dcl_{\R_{\exp}}(A)\). Theorem~\ref{thm:A} consequently gives
\(G_{\exp}\) and all the asserted dimension and chain properties,
including \(E(0)=\dcl_{\R_{\exp}}(\varnothing)\).
Each section is definably closed; definable choice for o-minimal
expansions of ordered fields \cite{vdD98} and the Tarski--Vaught test
show that it is an elementary substructure of \(\R_{\exp}\).
Its ordered-field reduct is therefore real closed.
\end{proof}

\section{Proof of Theorem~\ref{thm:B}}\label{sec:bases}

Fix a coded family of pregeometries as in
Definition~\ref{def:pregeometry-family}.
A label is a triple \(\lambda=(k,\gamma,s)\) with
\((k,\gamma)\in I\) and \(0\le s\le k\).
Write \(\cl_\lambda=\cl_{(k,\gamma)}\).
The set of labels is lightface \(\Delta^1_1\).
For a label \(\lambda\), a real \(q\), and \(x\in\R^k\), put
\[
 W(\lambda,q,x)
 \quad\Longleftrightarrow\quad
        r_q^k\ge s\ \hbox{or}\ x\notin N_k(q).
\]
This relation is arithmetic.
We build a basis for every label in a single recursion; one stage adds a
finite batch to at most one label.

\begin{lemma}\label{lem:basis-extension}
Let \(\lambda=(k,\gamma,s)\) be a label, let \(J\subseteq\R^k\) be
countable and \(\cl_\lambda\)-independent, and let \(u\in\omega^\omega\).
\begin{enumerate}[label=\textup{(\alph*)}]
\item If \(y\notin\cl_\lambda(J)\), there are distinct points
      \(x,x'\in U_k(s)\setminus J\) such that
      \(J\cup\{x,x'\}\) is independent,
      \(y\in\cl_\lambda(\{x,x'\})\),
      \(u\leqT x,\ u\leqT x'\), and \(x\equivh x'\).
\item For every \(q\in\omega^\omega\) there is
      \(x\in U_k(s)\setminus\cl_\lambda(J)\) such that
      \(u\leqT x\) and \(W(\lambda,q,x)\).
\end{enumerate}
\end{lemma}

\begin{proof}
Choose \(c_u\in2^\omega\) coding the graph of \(u\), so \(u\leqT c_u\).
For (a), use Proposition~\ref{prop:reflection-cantor} with
\[
             d=0,\qquad z=y,\qquad c=\Cut(y)\oplus c_u.
\]
It gives a perfect set \(K\subseteq[0,1]^k\).
By (C1), \(\cl_\lambda(J\cup\{y\})\) is countable, so choose
\[
                 x\in K\setminus\cl_\lambda(J\cup\{y\}),
                 \qquad x'=y-x.
\]
Then \(J\cup\{x\}\) is independent, and (C2) gives
\(y=x+x'\in\cl_\lambda(\{x,x'\})\).
If \(x'\in\cl_\lambda(J\cup\{x\})\), then
\(y\in\cl_\lambda(J\cup\{x\})\).
Since \(y\notin\cl_\lambda(J)\), exchange would imply
\(x\in\cl_\lambda(J\cup\{y\})\), a contradiction.
Thus \(J\cup\{x,x'\}\) is independent, in particular \(x\ne x'\) and
both points are outside \(J\).
The reflection proposition gives \(x,x'\in U_k(0)\subseteq U_k(s)\)
and makes both points compute \(u\) and \(y\).
Subtraction is arithmetic on cuts; hence \(x'\leqh x\) and
\(x\leqh x'\).

For (b), set \(d=r_q^k\) if \(r_q^k<s\), and \(d=0\) otherwise.
Use Proposition~\ref{prop:reflection-cantor} with parameters
\((k,d,0,c_u)\).
Its perfect set \(K\) lies in \(U_k(d)\subseteq U_k(s)\), and every
point computes \(u\).
If \(r_q^k\ge s\), choose any point of
\(K\setminus\cl_\lambda(J)\).
If \(r_q^k<s\), choose \(t\) with
\(\theta_q^k<t<r_q^k=d\).
Then
\[
 \Haus^t\bigl(N_k(q)\cup\cl_\lambda(J)\bigr)=0,
 \qquad
 \Haus^t(K)>0.
\]
Choose \(x\) outside that union.
In either case \(W(\lambda,q,x)\) holds and adjoining \(x\) preserves
independence.
\end{proof}

Within this section, call a real \(b\) null if \(b(0)=0\).
A nonnull \emph{record} has \(b(0)=m\ge1\) and
\[
 b^+=\widehat k\oplus\gamma\oplus\Cut(s)
              \oplus\Cut(x_1)\oplus\cdots\oplus\Cut(x_m),
\]
where \((k,\gamma,s)\) is a label and
\(\Cut(x_1)<_{\mathrm{lex}}\cdots<_{\mathrm{lex}}\Cut(x_m)\).
Write
\(\lab(b)=(k,\gamma,s)\) and \(\bat(b)=\{x_1,\ldots,x_m\}\).
The set \(\mathrm{Rec}\) of records is \(\Delta^1_1\).
Its label and entries are uniformly readable from the record.
For every label \(\lambda\) and finite nonempty batch \(J\), there is
exactly one such record. It is arithmetic in \(\lambda\) and any enumeration
of \(J\).

A \emph{demand} is a real \(p\) with \(p(0)=e\in\{0,1\}\) and
\[
              p^+=\widehat k\oplus\gamma\oplus\Cut(s)\oplus\pi
\]
for a label \(\lambda=(k,\gamma,s)\).
If \(e=0\), require \(\pi=\Cut(y)\) for some \(y\in\R^k\); this is the
spanning demand for \(y\).
If \(e=1\), let \(\pi=q\in\omega^\omega\); this is the dimension demand
for \(q\).
The set \(\mathrm{Dem}\) is \(\Delta^1_1\).
Every spanning and dimension demand occurs among the reals.

For a history \(a\), define
\[
 J_\lambda(a)=
 \bigcup\{\bat(b):b\in\ran a\cap\mathrm{Rec},\ \lab(b)=\lambda\}.
\]
Call \(a\) coherent if every entry is null or a record and each
\(J_\lambda(a)\) is independent.
These conditions depend only on \(\ran a\).
The sets \(J_\lambda(a)\) are countable and are empty unless the label
appears in \(a\).

Define \(\mathscr F_{\mathrm{bas}}(a,p,b)\) by the following alternatives:
\begin{enumerate}[label=\textup{(B\arabic*)}]
\item \(a\) is not coherent, with no restriction on \(p,b\);
\item \(a\) is coherent, \(p\notin\mathrm{Dem}\), and \(b\) is null;
\item \(a\) is coherent, \(p\in\mathrm{Dem}\) has label
      \(\lambda=(k,\gamma,s)\), and \(b\in\mathrm{Rec}\) has the same
      label, with the following properties, where
      \(J=J_\lambda(a)\) and \(L=\bat(b)\):
      \begin{enumerate}[label=\textup{(\alph*)}]
      \item \(L\subseteq U_k(s)\setminus J\), and \(J\cup L\) is independent;
      \item \(x\equivh x'\) for all \(x,x'\in L\), and
            \(\widehat k\oplus\gamma\oplus\Cut(s)\leqh\Cut(x)\)
            for every \(x\in L\);
      \item for a spanning demand, \(y\in\cl_\lambda(J\cup L)\);
            for a dimension demand, \(W(\lambda,q,x)\) holds for some \(x\in L\).
      \end{enumerate}
\end{enumerate}

\begin{lemma}\label{lem:basis-selection}
The relation \(\mathscr F_{\mathrm{bas}}\) is lightface \(\Pi^1_1\), and
every section \((\mathscr F_{\mathrm{bas}})_{a,p}\) is cofinal in the
hyperdegrees.
\end{lemma}

\begin{proof}
A position in \(a\) is a pair \((n,\ell)\) with \(n<|a|\),
\(a(n)\in\mathrm{Rec}\), and \(1\le\ell\le a(n)(0)\).
Its point and label are uniformly given by columns of \(a(n)\).
The set of positions is \(\Delta^1_1\), and equality of points or labels is
arithmetic.

For any set of positions given by a \(\Delta^1_1\) condition and having a
common label, its point set \(S\) is independent exactly when, for every
position with point \(x\) and every finite tuple of positions whose points
differ from \(x\), the point \(x\) does not belong to the closure of that
tuple. Finite character and (C3) make this condition \(\Delta^1_1\).
Likewise, \(y\in\cl_\lambda(S)\) is a number-existential condition over
finite tuples of positions and is \(\Delta^1_1\).
It follows that coherence is \(\Delta^1_1\).
The remaining requirements in (B3), except for the hyperarithmetic
reducibility conditions, are \(\Delta^1_1\). Those conditions are
\(\Pi^1_1\) by Lemma~\ref{fact:effective}(a).
All three alternatives therefore define a \(\Pi^1_1\) relation.

For cofinality, fix \(a,p,u\).
If \(a\) is incoherent, take \(b=u\).
If \(a\) is coherent but \(p\notin\mathrm{Dem}\), take \(b=0^\frown u\).
Otherwise let \(\lambda=(k,\gamma,s)\) be the label and
\(J=J_\lambda(a)\), and put
\[
             u^*=u\oplus\widehat k\oplus\gamma\oplus\Cut(s).
\]
Use \(u^*\) as the target real in each of the following applications
of Lemma~\ref{lem:basis-extension}.
For a spanning demand with \(y\notin\cl_\lambda(J)\), use the pair from
Lemma~\ref{lem:basis-extension}(a).
For a spanning demand already met, use any singleton from part (b) of
that lemma, with an arbitrary null-set code.
For a dimension demand use part (b) with its actual code \(q\).
In each case the canonical record for the resulting batch satisfies
(B3), and computes a point which computes \(u^*\), hence \(u\).
\end{proof}

\begin{proof}[Proof of Theorem~\ref{thm:B}]
Apply Theorem~\ref{fact:selection} to \(\mathscr F_{\mathrm{bas}}\).
Let \(Y\) be the resulting lightface \(\Pi^1_1\) set of records, with
enumerations \((p_\xi)_{\xi<\omega_1}\), \((b_\xi)_{\xi<\omega_1}\),
and histories \(a_\xi\) witnessing compatibility.

Every history \(a_\xi\) is coherent.
Indeed, assume this for all earlier histories.
Each earlier record is then null or satisfies (B3).
If \(S\subseteq J_\lambda(a_\xi)\) is finite and nonempty, choose finitely many earlier
batches covering \(S\), and let \(\eta<\xi\) be the last of their stages.
Then
\[
             S\subseteq J_\lambda(a_\eta)\cup\bat(b_\eta),
\]
an independent set by (B3)(a).
Thus every finite subset of \(J_\lambda(a_\xi)\) is independent, so the
whole set is independent by Lemma~\ref{lem:independence}.
This proves the induction.

For every label \(\lambda\), put
\[
 B_\lambda=
 \bigcup\{\bat(b):b\in Y\cap\mathrm{Rec},\ \lab(b)=\lambda\}
 =\bigcup_{\xi<\omega_1}J_\lambda(a_\xi).
\]
It is independent as the union of a chain of independent sets.
All its points lie in \(U_k(s)\), so
\(\dimH B_\lambda\le s\).
Every spanning demand is met at its stage, giving
\(\cl_\lambda(B_\lambda)=\R^k\).
Every dimension demand with \(r_q^k<s\) contributes a point outside
\(N_k(q)\). Corollary~\ref{cor:dimension-test} therefore yields
\(\dimH B_\lambda\ge s\).

It remains to check joint definability.
For \(v=\widehat k\oplus\gamma\oplus\Cut(s)\oplus\Cut(x)\), we have
\begin{equation}\label{eq:basis-bounded}
 x\in B_{(k,\gamma,s)}
 \quad\Longleftrightarrow\quad
 (\exists b\leqh v)\,
 [b\in Y\cap\mathrm{Rec},\
   \lab(b)=(k,\gamma,s),\ x\in\bat(b)].
\end{equation}
The right-to-left implication is immediate.
For the converse, choose a record whose batch contains \(x\).
By (B3)(b), every other point of that batch is hyperarithmetic in \(x\).
The canonical record is arithmetic in its label and the finite batch,
so it is hyperarithmetic in \(v\).
The matrix of \eqref{eq:basis-bounded} is \(\Pi^1_1\).
Lemma~\ref{fact:effective}(b) proves that the membership relation, jointly
in the label and the point, is lightface \(\Pi^1_1\).
This is the required set \(\mathcal G\).

If the family is given relative to a real \(\delta\), the same argument
works with \(\delta\) added as a parameter throughout: we use the
selection theorem relative to \(\delta\), and in the cofinality proof
of Lemma~\ref{lem:basis-selection} we use \(u^*\oplus\delta\) as the
target real. Replacing \(v\) by \(\delta\oplus v\) in
\eqref{eq:basis-bounded} and applying Lemma~\ref{fact:effective}(b)
relative to \(\delta\) gives \(\mathcal G\in\Pi^1_1(\delta)\).
\end{proof}

\begin{proof}[Proof of Corollary~\ref{cor:bases}]
For every \(n\ge1\), consider the pregeometry
\[
              \operatorname{lin}^n(A)=\spanop_{\Q}(A),
              \qquad A\subseteq\R^n,
\]
and, identifying \(\C\) with \(\R^2\), consider
\[
 \alg(A)=\{z\in\C:z\text{ is algebraic over }\Q(A)\},
 \qquad A\subseteq\C.
\]
Their bases are precisely the Hamel bases over \(\Q\) and the
transcendence bases of \(\C/\Q\), respectively.
Both closures contain sums of two input points and take countable
sets to countable sets: there are only countably many finite rational
linear combinations, and only countably many polynomials over a
countable field, each nonzero polynomial having finitely many roots.
Thus (C1) and (C2) hold.

For (C3), finite linear-span membership is given by
\[
 y\in\spanop_{\Q}\{x_1,\ldots,x_m\}
 \quad\Longleftrightarrow\quad
 (\exists r_1,\ldots,r_m\in\Q)\;
 y=\sum_{i=1}^{m}r_ix_i.
\]
Similarly, \(y\in\alg(\{x_1,\ldots,x_m\})\) if and only if there is
\(P\in\Q[X_1,\ldots,X_m,Y]\) such that
\[
              P(\bar x,Y)\ne0\text{ in }\C[Y],
              \qquad P(\bar x,y)=0,
\]
as follows by clearing denominators in \(\Q(\bar x)[Y]\).
Real and imaginary parts of complex rational polynomials are rational
polynomials in the real coordinates, so Lemma~\ref{lem:cuts} makes
these tests arithmetic, uniformly in the finite tuples and the
ambient dimension; the displayed quantifiers range over countable
sets with computable enumerations.

Let \(\gamma_0\in2^\omega\) be the constant zero sequence.
Apply Theorem~\ref{thm:B} to the linear-span family with label set
\(\{(n,\gamma_0):n\ge1\}\), and to the algebraic-closure family with
the single label \((2,\gamma_0)\).
Taking sections at the computable parameter \(\gamma_0\) gives the
required lightface \(\Pi^1_1\) relations \(H(n,s,x)\) and \(T(s,z)\),
whose sections are \(\Pi^1_1(s)\).
\end{proof}

\section{Remarks and questions}\label{sec:remarks}

\begin{remark}[Closure under compositions and computability]
\label{rem:closure-scope}
Condition (F3) applies to the entire composition-closed family.
On \(\R^2\), addition and \(T(x_1,x_2)=(x_1,-x_2)\) separately
satisfy the condition on unary sections, but the generated map
\(x\mapsto x+T(x)=(2x_1,0)\) fails (F3).

Condition (F3) cannot be omitted in general, as the following
zero--one law illustrates.
By \cite[Theorem~1.1 and Corollary~1.2]{Wan26}, a subring of \(\R\)
closed under binary interleaving for every choice of binary expansions
has Hausdorff dimension \(0\) or \(1\); so does the field of reals
computable from members of a Turing ideal. Thus, for \(0<\alpha<1\),
the subfields of Corollary~\ref{cor:subfields} cannot be closed under
binary interleaving or arise from Turing ideals. Condition (F2) only
requires effective evaluation of the specified operations.
\end{remark}

\begin{remark}[Bases without definability requirements]\label{rem:zfc-bases}
The existence of bases of prescribed Hausdorff dimension, without a
definability requirement, is already provable in ZFC.
For Hamel bases of \(\R\) over \(\Q\), this is the result of Lutz, Qi and
Yu \cite{LQY24}.
More generally, suppose \(\cl\) is a pregeometry on \(\R^k\) satisfying
\[
 a+b\in\cl(\{a,b\}),\qquad
 |\cl(A)|\le |A|+\aleph_0
 \quad(A\subseteq\R^k).
\]
Fix \(s\in[0,k]\), enumerate all spanning demands and all dimension
demands with \(r_q^k<s\) in order type \(2^{\aleph_0}\), and run the
construction of Section~\ref{sec:bases}, omitting the coding requirements.
At any stage the previously chosen independent set has size less than
\(2^{\aleph_0}\), and so does its closure.
For a spanning demand the available perfect set has size \(2^{\aleph_0}\).
For a dimension demand, the Borel set \(K\setminus N_k(q)\) has positive
\(\Haus^t\)-measure for some \(t>0\), hence is uncountable and has size
\(2^{\aleph_0}\) by the perfect set theorem for Borel sets.
Thus the necessary points can always be chosen.
The resulting basis has dimension \(s\).
The contribution of Theorem~\ref{thm:B} is its coanalytic definability
and simultaneous uniformity.

For a finite-character closure, countability of finite closures implies
the displayed cardinal bound. It applies to all the pregeometries in
Definition~\ref{def:pregeometry-family}.
The Euclidean function-closure construction in Theorem~\ref{thm:A} has
a different requirement: it controls all unary vector sections over
every finite tuple of earlier generators. Its simultaneous Cantor
argument uses the countability of that earlier parameter set.
\end{remark}

\begin{remark}[A contrasting consistency statement]\label{rem:solovay}
Relative to an inaccessible cardinal, it is consistent with ZFC that
there are no projective proper subfields of \(\R\) of positive Hausdorff
dimension and no projective bases of the kinds in
the Hamel and transcendence applications of Theorem~\ref{thm:B}.
Indeed, let \(\kappa\) be inaccessible and force with
\(\operatorname{Coll}(\omega,{<}\kappa)\).
In the resulting extension, the Hausdorff-measure regularity results of
Farah and Zapletal \cite[Corollaries~5.12 and~5.21]{FZ06}
give a compact positive-dimensional subset of every positive-dimensional
set of reals definable from real and ordinal parameters.
If such a set is a subring, the ring generated by this compact subset
is a \(\sigma\)-compact positive-dimensional subring, and is therefore
\(\R\) by \cite{EM03}.
To exclude a real-and-ordinal definable Hamel basis \(B\) of \(\R^n\)
over a countable field \(K\), fix \(b\in B\) and put
\(H=\spanop_K(B\setminus\{b\})\). This is a real-and-ordinal definable
additive subgroup, and its cosets \(qb+H\), \(q\in K\), cover \(\R^n\).
Thus \(H\) is nonmeager. Solovay's theorem \cite{Sol70} gives it the
Baire property, and Pettis' theorem \cite[Theorem~9.9]{Kec95} makes it
open. The connectedness of \(\R^n\) then forces \(H=\R^n\), contrary
to the independence of \(B\).
The transcendence-basis assertion follows from
\cite[Theorem~12.3.12]{LZ20}, using the trivial forcing over the
associated symmetric Solovay inner model, which has the same reals and
contains the real-and-ordinal definable sets.
Every projective set is real-and-ordinal definable.

Consequently, assuming the consistency of an inaccessible cardinal,
the conclusions of the two main theorems cannot be proved in ZFC
without their constructibility hypothesis.
The same applies to Corollary~\ref{cor:subfields} and
Corollary~\ref{cor:exponential}, since they concern proper subfields of
\(\R\) of positive Hausdorff dimension.
\end{remark}

\begin{remark}[Thin sections]\label{rem:thin}
A set is thin if it contains no nonempty perfect subset.
Under \(V=L\), Vidny\'anszky's method can select the record set inside
the self-constructible reals
\[
                  \mathcal S=\{a\in\omega^\omega:a\in L_{\omega_1^a}\},
\]
which form a thin set \cite[Section~5.2]{Vid14},
\cite[Theorem~(2A-1)]{Kec75}.
For the bases of Theorem~\ref{thm:B}, condition (B3)(b) makes every
point of a batch hyperarithmetically compute its label, and makes all
points in the batch hyperarithmetically equivalent. Hence the canonical
record and each point in its batch are hyperarithmetically equivalent,
for arbitrary labels \((k,\gamma,s)\). Since self-constructibility is
invariant under hyperarithmetic equivalence, the cuts of all points
in every basis section belong to \(\mathcal S\).

The sets of Theorem~\ref{thm:A} can be made thin for every label
\(\iota=(n,\gamma)\) as well. Let
\(y\in A_\iota(\alpha)\setminus\cl_\iota(\varnothing)\), and let
\(b_\xi\) be the last record in a minimal support of \(y\), with point
\(z\). Then \(b_\xi\leqh\Cut(y)\) by Lemma~\ref{cl:finite-support}.
Conversely, since \(z\notin\cl_\iota(\varnothing)\), (P3)(c) applied
to the empty tuple and to \(y=z\) shows that the whole history
\(A_\xi\) is hyperarithmetic in \(\Cut(z)\), hence in \(b_\xi\).
So are all the generators in the support of \(y\). The label is also
read from \(b_\xi\), so the \(\Delta^1_1\) evaluation relation gives
\(\Cut(y)\leqh b_\xi\). Thus \(\Cut(y)\equivh b_\xi\), and the cuts of
\(A_\iota(\alpha)\setminus\cl_\iota(\varnothing)\) belong to
\(\mathcal S\).
The map \(x\mapsto\Cut(x)\) is Borel and injective, so the image of a
perfect subset of \(\R^n\) would be an uncountable analytic set, and
would contain a perfect set \cite{Kec95}. Hence
\(A_\iota(\alpha)\setminus\cl_\iota(\varnothing)\) is thin, and so is
\(A_\iota(\alpha)\), since \(\cl_\iota(\varnothing)\) is countable.
The same cut-map argument proves that every basis section is thin.
In particular, \(A_\iota(n)\) is a thin set of full Hausdorff
dimension \(n\) for every label \(\iota=(n,\gamma)\).
The argument relativizes to a fixed real parameter by using the
corresponding relative self-constructible reals.
\end{remark}

\begin{question}\label{q:perfect}
	Assume \(V=L\).
	Can the coanalytic Hamel and transcendence bases above be chosen to
	contain a nonempty perfect subset while retaining any prescribed
	Hausdorff dimension in their respective ranges?
	Can proper subfields of \(\R\), or real closed ones, be so chosen at
	every positive dimension?
	Compare \cite[Problem~5.8]{Vid14}.
\end{question}

\begin{question}\label{q:from-one}
	Over ZFC, does the existence of a coanalytic Hamel basis of \(\R\)
	over \(\Q\) imply the existence of such bases of every Hausdorff dimension
	in \([0,1]\)?
	Does the existence of a coanalytic proper subfield of \(\R\) of positive
	Hausdorff dimension imply that every \(\alpha\in(0,1]\) is realized by
	a proper coanalytic subfield?
\end{question}

\begin{remark}[The thin basis case]
	The Hamel-basis part of Question~\ref{q:from-one} has an affirmative
	answer if the given coanalytic basis is thin.
	If \(B\) is a thin \(\Pi^1_1(a)\) Hamel basis, the relativized
	Mansfield--Solovay theorem \cite[Theorem~7]{Man70} yields \(B\subseteq L[a]\).
	Hence \(\R=\spanop_{\Q}(B)\subseteq L[a]\).
	The selection method and the proofs relativize to the constructible
	universe with parameter \(a\). Apply them inside \(L[a]\).
	That model and the ambient universe have the same reals.
	The coanalytic definitions and the algebraic properties are therefore
	absolute, and the Hausdorff dimensions agree because countable
	rational-cube covers are coded by reals.
	This yields coanalytic Hamel bases of every dimension, and also the
	field constructions above, allowing the parameter \(a\).
	The same argument applies to thin coanalytic transcendence bases of
	\(\C\) over \(\Q\).
\end{remark}

\begin{question}\label{q:measure}
Assume \(V=L\). For every \(n\ge1\), every countable subfield
\(K\subseteq\R\), and every \(0<s<n\), is there a coanalytic Hamel
basis \(B\) of \(\R^n\) over \(K\) such that
\[
                    \dimH B=s,\qquad 0<\Haus^s(B)<\infty?
\]
\end{question}

By left continuity, the sets in Theorem~\ref{thm:A} already satisfy
\(\Haus^\alpha(A_\iota(\alpha))=0\) for every \(0<\alpha\le n\), since
each rational section in the defining union has dimension below \(\alpha\).

For subrings, finite positive Hausdorff measure at a positive exponent
is impossible: well-separated translates of a bounded positive-measure
part would force infinite measure.
For a coanalytic algebraically independent set \(B\subseteq\C\), the
same requirement would give a positive-dimensional compact subset
\(P\subseteq B\), by universal measurability and inner regularity on a
finite-measure Borel hull. The ring \(\Z[P]\) is then \(\sigma\)-compact,
so \cite[Theorem~2]{EM03} would force \(\Z[P]=\R\) or \(\C\), contrary
to \(1/2\notin\Z[P]\).

\begin{question}\label{q:packing}
Assume \(V=L\). For each of the following kinds of sets, which pairs
\(0\le s\le t\le d\) occur as the Hausdorff and packing dimensions
of a coanalytic example: Hamel bases of \(\R\) over \(\Q\), proper
subrings of \(\R\), or transcendence bases of \(\C\) over \(\Q\),
where \(d=1\) for the first two kinds and \(d=2\) for the third?
\end{question}

\subsection*{Acknowledgements}
In preparing this manuscript, the author used generative AI tools to
refine the exposition, explore mathematical ideas, develop and check
proofs, and identify relevant literature.
In particular, the Danus system was used in the development of the
coded Cantor-set constructions in Propositions~\ref{prop:general-cantor}
and~\ref{prop:reflection-cantor}.
The author takes full responsibility for the manuscript.

\end{document}